\documentclass[11pt, a4paper]{amsart}

\usepackage[utf8]{inputenc}
\usepackage[T1]{fontenc}
\usepackage{amsmath}
\usepackage{amsfonts}
\usepackage{ae}
\usepackage{units}
\usepackage{icomma}
\usepackage{color}
\usepackage{graphicx}
\usepackage{bbm}
\usepackage{amsthm}
\usepackage{amssymb}
\usepackage{cancel}
\usepackage{fullpage}
\usepackage{upgreek}
\usepackage{algorithm}
\usepackage{algorithmic}
\usepackage{type1cm}
\usepackage{eso-pic}
\usepackage[breaklinks,pdfpagelabels=false]{hyperref}
\usepackage{verbatim}
\usepackage{tikz}
\usetikzlibrary{arrows}

\newcommand{\abs}[1]{\ensuremath{\left| #1 \right|}}

\newcommand{\bm}[1]{\mbox{\boldmath $ {#1} $}}

\renewcommand{\epsilon}{\varepsilon}

\newtheorem{theorem}{Theorem}[section]
\newtheorem{lemma}[theorem]{Lemma}
\newtheorem{proposition}[theorem]{Proposition}
\newtheorem{corollary}[theorem]{Corollary}

\newtheorem{conjecture}[theorem]{Conjecture}
\theoremstyle{definition}
\newtheorem{definition}[theorem]{Definition}
\newtheorem{remark}[theorem]{Remark}
\newtheorem{question}[theorem]{Question}

\numberwithin{equation}{section}
\numberwithin{theorem}{section}

\begin{document}

\title[The ``no justice in the universe'' phenomenon: Why honesty of effort
  may not be rewarded in tournaments] 
      {The ``no justice in the universe'' phenomenon: Why honesty of effort
        may not be rewarded in tournaments}


\author{Peter Hegarty} \address{Department of Mathematical Sciences, 
Chalmers University Of Technology and University of Gothenburg,
41296 Gothenburg, Sweden} \email{hegarty@chalmers.se}

\author{Anders Martinsson} \address{Department of Mathematical Sciences,
Chalmers University Of Technology and University of Gothenburg,
41296 Gothenburg, Sweden} \email{andemar@chalmers.se}

\author{Edvin Wedin} \address{Department of Mathematical Sciences,
Chalmers University Of Technology and University of Gothenburg,
41296 Gothenburg, Sweden} \email{edvinw@chalmers.se}


\date{\today}

\begin{abstract}
  In 2000 Allen Schwenk, using a well-known mathematical model of matchplay
  tournaments in which the probability of one player beating another in a
  single match is fixed for each pair of players, showed that the classical
  single-elimination, seeded format can be ``unfair'' in the sense that
  situations can arise where an indisputibly better (and thus higher seeded)
  player may have a smaller probability of winning the tournament than a worse
  one. This in turn implies that, if the players are able to influence their
  seeding in some preliminary competition, situations can arise where it is in
  a player's interest to behave ``dishonestly'', by deliberately trying to
  lose a match. This motivated us to ask whether it is possible for a tournament
  to be both honest, meaning that it is impossible for a situation to arise
  where a rational player throws a match, and ``symmetric'' - meaning
  basically that the rules treat everyone the same - yet unfair, in the sense
  that an objectively better player has a smaller probability of winning than
  a worse one. After rigorously
  defining our terms, our main result is that such tournaments exist and we
  construct explicit examples for any number $n \geq 3$ of players. For
  $n=3$, we show (Theorem \ref{thm:threeplayer}) that the collection of
  win-probability vectors for such tournaments form a $5$-vertex
  convex polygon in $\mathbb{R}^3$, minus some boundary points. We conjecture
  a similar result for any
  $n \geq 4$ and prove some partial results towards it. 
\end{abstract}


\maketitle



\section{Introduction}\label{sect:intro}


In their final game of the group phase at the 2006 Olympic ice-hockey
tournament, a surprisingly lethargic Swedish team lost $3-0$ to Slovakia. The
result meant they finished third in their group, when a win would have
guaranteed at worst a second placed finish. As the top four teams in each of
the two groups qualified for the quarter-finals, Sweden remained in the
tournament after this abject performance, but with a lower seeding for the
playoffs. However, everything turned out well in the end as they crushed both
their quarter- and semi-final opponents ($6-2$ against Switzerland and $7-3$
against the Czech Republic respectively), before lifting the gold after a
narrow $3-2$ win over Finland in the final.

The Slovakia match has gained notoriety because of persistent rumours
that Sweden threw the game in order to avoid ending up in the same half of the
playoff draw as Canada and Russia, the two traditional giants of ice-hockey.
Indeed, in an interview in 2011, Peter Forsberg, one of Sweden's top stars,
seemed to admit as much{\footnote{See
    \texttt{www.expressen.se/sport/hockey/tre-kronor/forsberg-slovakien-var-en-laggmatch}}},
though controversy remains about the proper interpretation of his words.
Whatever the truth in this regard, it certainly seems as though Sweden were
better off having lost the game.
\par Instances like this in high-profile sports
tournaments, where a competitor is accused of deliberately losing a game, are
rare and tend to attract a lot of attention when they occur.
This could be considered surprising given that deliberate underperformance
in sport is nothing unusual. For example, quite often a team will decide to
rest their best players or give less than $100 \%$ effort when faced with an
ostensibly weaker opponent, having calculated that the risk in so doing is
outweighed by future potential benefits. Note that this could occur even in a
single-elimination knockout tournament, with a team deciding to trade an
elevated risk of an early exit for higher probability of success later on. Of
course, in such a tournament it can never be in a team's interest to actually
\emph{lose}. However, many tournaments, including Olympic ice-hockey, are
based on the template of two phases, the first being a round-robin event
(everyone meets everyone) which serves to rank the teams, and thereby provide
a seeding for the second, knockout phase{\footnote{In 2006, the Olympic ice
    hockey tournament employed a minor modification of this template. There
    were $12$
    teams. In the first phase, they were divided into two groups of six, each
    group playing round-robin. The top four teams in each group qualified for
    the knockout phase. The latter employed standard seeding (c.f. Figure
    \ref{fig:knockout}), but with the extra condition that teams from the same
    group could not meet in the quarter-finals. This kind of modification of the
    basic two-phase template, where the teams are first divided into smaller
    groups, is very common since it greatly reduces the total number of matches
    that need to be played.} Teams are incentivized to
  perform well in the first phase by $(1)$ often, only higher
  ranking teams qualify for the second phase, and $(2)$ standard seeding (c.f.
  Figure \ref{fig:knockout}) aims to place high ranking teams far apart in the
  game tree, with higher ranking teams closer to lower ranking ones, meaning
  that a high rank generally gives you an easier starting position.
\par
The example of Sweden in 2006 illustrates the following phenomenon of
two-phase tournaments. Since a weaker
team always has a non-zero probability of beating a stronger one in a single
match, a motivation to throw a game in the first phase can arise when it seems
like the ranking of one's potential knockout-phase opponents does not reflect
their actual relative strengths. Sweden's loss to Slovakia meant they faced
Switzerland instead of Canada in the quarter-final and most observers would
probably have agreed that this was an easier matchup, despite Switzerland
having finished second and Canada third in their group (Switzerland also beat
Canada $2-0$ in their group match).

The above phenomenon is easy to understand and begs the fascinating question
of why instances of game-throwing seem to be relatively rare. We don't
explore that (at least partly psycho-social) question further in this paper.
However, \emph{even if} game-throwing is rare, it is still certainly a
weakness of this tournament format that situations can arise where a team is
given the choice between either pretending to be worse than they are, or
playing \emph{honestly} at the cost of possibly decreasing their chances of
winning the tournament.

In a 2000 paper \cite{Sch}, Allan Schwenk studied the question of how to best
seed a knockout tournament from a mathematical point of view. One, perhaps
counter-intuitive, observation made in that paper is that standard seeding
does not necessarily benefit a higher-ranking players, even when the ranking
of its potential opponents \emph{accurately} reflects their relative strengths. 
Consider a matchplay tournament with $n$ competitors, or ``players'' as
we shall
henceforth call them, even though the competitors may be teams. In Schwenk's
mathematical model, the players are numbered $1$ through $n$ and
there are fixed probabilities $p_{ij} \in [0,\,1]$ such
that, whenever
  players $i$ and $j$ meet, the probability that $i$ wins is $p_{ij}$.
  Draws are not allowed, thus $p_{ij} + p_{ji} = 1$. Suppose
  we impose the conditions
  \par (i) $p_{ij} \geq 1/2$ whenever $i < j$,
  \par (ii) $p_{ik} \geq p_{jk}$ whenever $i<j$ and $k\not\in \{i, j\}$.
  \\
  Thus, for any $i < j$, $i$ wins against $j$ with probability at least $1/2$,
  and for any other player $k$, $i$ has at least as high a probability of
  beating $k$ as $j$ does. 
  It then seems unconstestable to
  assert that player $i$ is at least as good as player $j$ whenever
  $i < j$. Indeed, if we imposed strict inequalities in (i) and (ii) we would
  have an unambiguous ranking of the players: $i$ is better than $j$ if and
  only if $i < j$. This is a very natural model to work with.
  It is summarized by a
  so-called \emph{doubly monotonic} $n \times n$
  matrix $M = (p_{ij})$, whose entries equal
  $\frac{1}{2}$ along the main diagonal, are non-decreasing from left to
  right along each row, non-increasing from top to bottom along each
  column and satisfy $p_{ij} + p_{ji} = 1$ for all $i,\,j$. We shall refer to
  the model as the \emph{doubly monotonic model (DMM) of tournaments}. It is
  the model employed throughout the rest of the paper. 
  \par In \cite{Sch}, Schwenk gave an example of an $8 \times 8$ doubly
  monotonic matrix such that, if the standard seeding method (illustrated in
Figure \ref{fig:knockout}) were employed
  for a single-elimination tournament, then player $2$ would have a higher
  probability of winning than player $1$.
\begin{figure*}[ht!]
  \includegraphics[width=.8\textwidth]{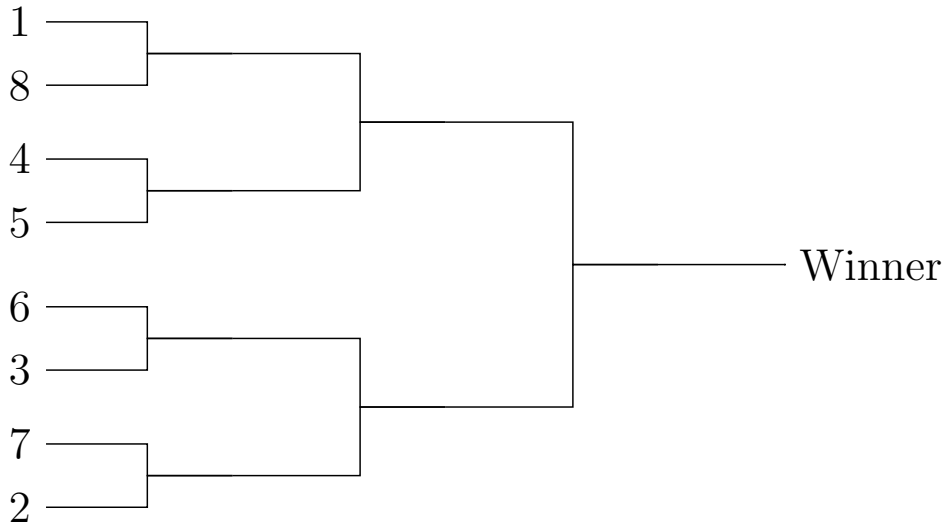} 
  \caption{The standard seeding for a single-elimination knockout tournament
    with $2^3 = 8$ players. In general, if there are $2^n$ players and the
    higher ranked player wins every match then, in the $i$th round,
    $1 \leq i \leq n$, the pairings will be
    $\{j+1,\,2^{n+1-i}-j\}$, $0 \leq j < 2^{n+1-i}$.}
\label{fig:knockout}
\end{figure*}
  As an evident corollary, assuming the
  same mathematical model one can concoct situations in two-phase tournaments
  of the kind considered above in which it is a player's interest to lose a
  game in the first phase even when, say, in every other match played to that
  point, the better team has won.

  Many tournaments consist of only a single phase, either
  round-robin{\footnote{or, more commonly, a \emph{league} format, where each
      pair meet twice.}} or single-elimination. As opposed to the
  aforementioned two-phase format, here it is not hard to see that it can
  never be in a team's interest to lose a game. Indeed, this is
  clear for the single-elimination format, as one loss means you're out of the
  tournament. In the round-robin format, losing one game, all else being
  equal, only decreases your own total score while increasing the
  score of some other team. As Schwenk showed,
  the single-elimination option, with standard seeding, may still not be fair,
  in the sense of always giving a higher winning probability to a better
  player. The obvious way around this is to randomize the draw. Schwenk
  proposed a method called
  \emph{Cohort Randomized Seeding}{\footnote{It is easy to see that the
      standard method cannot result in a player from a lower cohort, as that
      term is defined by Schwenk, having a higher probability of winning the
      tournament than one in a higher cohort.}}, which seeks to respect the
  economic incentives
  behind the standard method{\footnote{The standard format
      ensures the romance of
  ``David vs. Goliath'' matchups in the early rounds, plus the likelihood of
  the later rounds featuring contests between the top stars, when public
  interest is at its highest. Schenk used the
      term \emph{delayed confrontation} for the desire to keep the top ranked
      players apart in the early rounds.}} while introducing just enough
  randomization to ensure that this basic criterion for fairness is satisfied.
  According to Schwenk himself, in email correspondence with us, no major
  sports competition has yet adopted his proposal{\footnote{On the other hand, uniformly random draws are commonly employed. An example is the English FA Cup, from the round-of-$64$ onwards.}}.

Even tournaments where it is never beneficial to lose a match often include another source
  of unfairness, in that players may face quite different schedules, for
  reasons of geography, tradition and so on. For example, qualifying for the
  soccer World Cup is organized by continent, an arrangement that effectively
  punishes European teams. The host nation automatically qualifies
  for the finals
  and is given a top seeding in the group phase, thus giving it an unfair
  advantage over everyone else. In the spirit of fair competition, one would
  ideally wish for a tournament not to give certain players any special
  treatment from the outset, and only break this symmetry after seeing how the
  teams perform within the confines of the tournament. Note that a
  single-elimination tournament with standard seeding is an example of such
  ``asymmetric scheduling'', unless the previous performances upon which the
  seeding is founded are considered part of the
  tournament.  
  \\ \par The above considerations lead us to the question on which this paper
  is based. Suppose the rules of a tournament ensure both
  \par - \emph{honesty}, meaning it is impossible for a situation to arise
  where it is in a player's interest to lose a game, and
  \par - \emph{symmetry}, meaning that the rules treat all the players equally.
  In particular, the rules should not depend on the identity of the players, or the order in which they entered the tournament.  \\
  \par Must it follow that the tournament is \emph{fair}, in the sense that a better
  player always has at least as high a probability of winning the tournament as
  a worse one? Having defined our terms precisely we will show below that the
  answer, perhaps surprisingly, is no. Already for three players, we will
  provide simple examples of tournaments which are symmetric and honest, but
  not fair. The question of ``how unfair'' a symmetric and honest tournament
  can be seems to be non-trivial for any $n\geq 3$ number of players.
  For $n=3$ we solve this problem exactly, and for $n\geq 4$ we formulate a
  general conjecture. The rest of the paper is organized as follows: 
  
  \begin{itemize}
    \item
  Section \ref{sect:defi} provides rigorous definitions. We will define what
  we mean by a
    \emph{(matchplay) tournament} and what it
    means for a tournament to be either \emph{symmetric}, \emph{honest} or
    \emph{fair}. The DMM is assumed throughout.
  \item
    Sections
    \ref{sect:threeplayer} and \ref{sect:nplayer} are the heart of the
    paper. In the former, we consider
    $3$-player tournaments and describe what appear to be the simplest possible
    examples of tournaments which are symmetric and honest, but not fair.
    Theorem \ref{thm:threeplayer} gives a precise characterization of
    those
    probability vectors $(x_1, \, x_2, \, x_3) \in \mathbb{R}^3$ which can arise
    as the vectors of win-probabilities for the players in a symmetric and
    honest tournament{\footnote{These results may remind some readers of the
        notion of a \emph{truel} and of the known fact that, in a truel, being
        a better shot does not guarantee a higher probability of winning (that
        is, of surviving). See \texttt{https://en.wikipedia.org/wiki/Truel}.
        Despite the analogy, we're not aware of any deeper connection between
        our results and those for truels, nor between their respective
        generalizations to more than three ``players''.}}; here
    fairness would mean $x_1 \geq x_2 \geq x_3$. 
  \item
    In Section \ref{sect:nplayer} we extend these ideas to a general
    method for constructing symmetric, honest and unfair $n$-player
    tournaments. We introduce a family of $n$-vertex digraphs
    and an associated convex polytope $\mathcal{A}^{*}_{n}$
    of probability vectors
    in $\mathbb{R}^n$ and show that every
    interior point of this polytope arises as the vector of
    win-probabilities of some symmetric
    and honest $n$-player tournament. The polytope $\mathcal{A}^{*}_{n}$
    includes all probability vectors satisfying
    $x_1 \geq x_2 \geq \dots \geq x_n$, but is shown to have a total of
    $\frac{3^{n-1} + 1}{2}$ corners, thus yielding a plethora
    of examples of symmetric and honest, but unfair tournaments. Indeed, we
    conjecture (Conjecture \ref{conj:nplayer}) that the
    vector of win-probabilities of \emph{any} symmetric and honest
    $n$-player tournament lies in $\mathcal{A}^{*}_{n}$.
  \item
    Section \ref{sect:frugal} considers the notion of a \emph{frugal}
    tournament, namely one which always begins by picking one player uniformly
    at random to take no further part in it (though he may still win).
    The tournaments constructed in Sections
    \ref{sect:threeplayer} and \ref{sect:nplayer} have this
    property, and the main result of Section \ref{sect:frugal} is, in
    essence, that frugal tournaments provide no counterexamples to
    Conjecture \ref{conj:nplayer}.
  \item
    Section \ref{sect:maps} introduces the notion of a \emph{tournament map}, which is a natural way to view tournaments as continuous functions. We describe its relation to the regular tournament concept. Using this, we show (Corollary \ref{cor:andersnormalform}) that any symmetric and honest tournament can be approximated arbitrarily well by one of the form described in the section. We further provide three applications. \par
    - The first is to \emph{strictly} honest tournaments, which means, informally, that a player should always be strictly better off in winning a match than in losing it. We show that any symmetric and honest tournament can be approximated arbitrarily well by a strictly honest one.
  \par
    - The second application is to \emph{tournaments with rounds}. For simplicity, we assume in the rest of the article that matches in a tournament are played one-at-a-time, something which is often not true in reality. Extending the notion of
    honesty to tournaments with rounds provides some technical challenges, which are discussed here.
  \par
    - The final application is to
    prove that the possible vectors of win-probablities for symmetric and
    honest $n$-player tournaments form a finite union of convex polytopes
    in $\mathbb{R}^n$, minus some boundary points.
    This provides, in particular, 
    some further evidence in support of Conjecture \ref{conj:nplayer}.
    \item
    In Section \ref{sect:futile}, we consider the 
    concept of a \emph{futile}
    tournament, one in which a player's probability of
    finally winning is never affected by whether they win or lose a given
    match. We prove that, in a symmetric and futile $n$-player
    tournament, everyone has probability $1/n$ of winning. This is exactly as
    one would expect, but it doesn't seem to be a completely trivial task to
    prove it.
  \item
    Finally, 
    Section \ref{sect:final} casts a critical eye on the various
    concepts introduced in the paper, and mentions some further possibilities
    for future work.
    \end{itemize}
      
    \section{Formal Definitions}\label{sect:defi}
    
    The word \emph{tournament} has many different meanings. In graph theory, it
    refers to a directed graph where, for every pair of vertices $i$ and $j$,
    there is an arc going either from $i$ to $j$ or from $j$ to $i$. In more
    common language, a \emph{matchplay} refers to a competition
    between a (usually
    relatively large) number of competitors/players/teams in which a winner is
    determined depending on the outcome of a number of individual matches, each
    match involving exactly two competitors. We concern ourselves
    exclusively with matchplay tournaments{\footnote{Athletics, golf, cycling,
        skiing etc. are examples of sports in which competitions traditionally
        take a different form, basically
        ``all-against-all''.}}. 
    Even with this restriction, the word ``tournament'' itself can be used to
    refer to: a \emph{reoccurring
      competition} with a fixed name and fixed format, such as the Wimbledon
    Lawn Tennis Championships, a \emph{specific instance} of a (potentially
    reoccurring)
    competition, such as the 2014 Fifa World Cup, or a
    specific \emph{set of rules} by which such a competition is structured,
    such as ``single-elimination knock-out with randomized seeding'',
    ``single round-robin with randomized scheduling'', 
    etc. We will here use tournament in this last sense.

    More precisely, we consider an \emph{$n$-player tournament} as a set of
    rules for how to arrange matches between $n$ \emph{players}, represented
    by numbers from $1$ to $n$. The decision on which players should meet each
    other in the next match may depend on the results from earlier matches as
    well as additional randomness (coin flips etc.). Eventually, the
    tournament should announce one of the players as the winner. We assume that:
\begin{enumerate}
\item A match is played between an (unordered) pair of players $\{i, j\}$. The
  outcome of said match can either be \emph{$i$ won}, or \emph{$j$ won}. In
  particular, no draws are allowed, and no more information is given back to
  the tournament regarding e.g. how close the match was, number of goals scored
  etc.
\item Matches are played sequentially one-at-a-time. In practice, many
  tournaments consist of ``rounds'' of simultaneous matches. We'll make some
  further remarks on this restriction in Subsection \ref{subsect:rounds}.
\item There is a bound on the number of matches that can be played in a
  specific tournament. So, for example, for three players we would not allow
  ``iteration of round-robin until someone beats the other two''. Instead,
  we'd require the tournament to break a potential three-way tie at some
  point, e.g. by randomly selecting a winner.
\end{enumerate}
Formally, we may think of a tournament as a randomized algorithm which is
given access to a function \texttt{PlayMatch} that takes as input an unordered
pair of numbers between $1$ and $n$ and returns one of the numbers.

In order to analyze our tournaments, we will need a way to model the outcomes
of individual matches. As mentioned in the introduction, we will here
employ the same simple model as Schwenk \cite{Sch}. For each pair of players
$i$ and $j$, we assume that there is some unchanging probability $p_{ij}$ that
$i$ wins in a match between them. Thus, $p_{ij} + p_{ji} = 1$ by (1) above.
We set $p_{ii} = \frac{1}{2}$ and denote the set of all 
possible $n \times n$ matrices by
$$\mathcal{M}_n = \{P\in [0, 1]^{n\times n} : P+P^T=\mathbf{1}  \},$$
where $\mathbf{1}$ denotes the all ones matrix. We say that
$P=(p_{ij})_{i,j\in[n]}\in\mathcal{M}_n$ is \emph{doubly monotonic} if $p_{ij}$ is
decreasing in $i$ and increasing in $j$. We denote
$$\mathcal{D}_n = \{P \in \mathcal{M}_n : P\text{ is doubly monotonic}\}.$$

We will refer to a pair $\bm{\mathcal{T}}=(\bm{T}, P)$ consisting
of an $n$-player
tournament $\bm{T}$ and a matrix $P\in\mathcal{M}_n$ as a \emph{specialization of
  $\bm{T}$}. Note that any such specialization defines a random process where
alternatingly two players are chosen according to $\bm{T}$ to play a match,
and the winner of the match is chosen according to $P$. For a given
specialization $\bm{\mathcal{T}}$ of a tournament, we let $\pi_k$ denote the
probability for player $k$ to win the tournament, and define the \emph{win
  vector} $\bm{wv}(\bm{\mathcal{T}})=(\pi_1, \dots, \pi_n)$. For a fixed
tournament $\bm{T}$ it will sometimes be useful to consider these
probabilities as functions of the matrix $P$, and we will hence write
$\pi_k(P)$ and
$\bm{wv}(P)$ to denote the corresponding probabilities in the specialization
$(\bm{T}, P)$

We are now ready to formally define the notions
    of symmetry, honesty and fairness.
    \\
    \\
      {\sc Symmetry:} Let $\bm{T}$ be an $n$-player tournament. For any
      permutation $\sigma\in\mathcal{S}_n$ and any $P\in \mathcal{M}_n$, we
      define $Q=(q_{ij})\in\mathcal{M}_n$ by $q_{\sigma(i)\sigma(j)} = p_{ij}$ for
      all $i, j \in [n]$. That is, $Q$ is the matrix one obtains from $P$
      after renaming each player $i\mapsto \sigma(i)$. We say that $\bm{T}$
      is \emph{symmetric} if, for any
      $P\in \mathcal{M}_n$, $\sigma\in\mathcal{S}_n$ and any $i\in[n]$, we have
      $\pi_i(P)=\pi_{\sigma(i)}(Q).$
      
      This definition is meant to capture the intuition that the rules
      ``are the same for everyone''. Note that any tournament can be turned
      into a symmetric one by first randomizing the order of the players.
      \\
      \\
        {\sc Honesty:} Suppose that a tournament $\bm{T}$ is in a state where
        $r\geq 0$ matches have already been played, and it just announced a
        pair of players $\{i, j\}$ to meet in match $r+1$. Let $\pi_i^+(P)$
        denote the probability that $i$ wins the tournament conditioned on the
        current state and on $i$ being the winner of match
        $r+1$, assuming the outcome of any subsequent match is decided according
        to $P\in\mathcal{M}_n$. Similarly, let $\pi_i^-(P)$ denote the
        probability that $i$ wins the tournament given that $i$ is the loser
        of match $r+1$. We say that $\bm{T}$ is \emph{honest} if, for any
        possible such state of $\bm{T}$ and any $P\in\mathcal{M}_n$, we have
        $\pi_i^+(P)\geq \pi_i^-(P)$.
        \par
        The tournament is said to be \emph{strictly honest} if in addition,
        for all $P\in\mathcal{M}^{o}_{n}$, the above inequality is strict, and
        all pairs of players have a positive probability to meet at least once
        during the tournament. Here $\mathcal{M}^{o}_{n}$ denotes the set of
        matrices $(p_{ij}) \in \mathcal{M}_n$ such that
        $p_{ij} \not\in \{0,\,1\}$. It makes sense to exclude these
        boundary elements
        since, if $p_{ij} = 0$ for every $j \neq i$, then player
        $i$ cannot affect his destiny at all. For instance, it seems natural
        to consider a single-elimination tournament as strictly honest, but
        in order for winning to be strictly better than losing, each player
        must retain a positive probability of winning the tournament whenever
        he wins a match.
        \par To summarize, in an honest tournament a player can never be put in
        a strictly better-off position by throwing a game. In a strictly
        honest tournament, a player who throws a game is always put in a
        strictly worse-off position.
        
        \begin{remark}\label{rem:state}
          We note that the ``state of a
          tournament'' may contain more information than what the players can
          deduce from the matches played so far. For instance, the two-player
          tournament that plays one match and chooses the winner with
          probability
          $0.9$ and the loser with probability $0.1$ is honest if the decision
          of whether to choose the winner or loser is made \emph{after} the
          match. However, if the decision is made \emph{beforehand}, then with
          probability $0.1$ we would have $\pi_1^+=\pi_2^+=0$
          and $\pi_1^-=\pi_2^-=1$. Hence, in this case the tournament is not
          honest.
          \end{remark} $\;$ \\
        {\sc Fairness:} Let $\bm{T}$ be an $n$-player tournament. We say
          that $\bm{T}$ is \emph{fair} if
          $\pi_1(P)\geq \pi_2(P) \geq \dots \geq \pi_n(P)$ for all
          $P\in\mathcal{D}_n$.
\\
\par The main purpose of the next two sections is to show that there exist
symmetric and honest tournaments which are nevertheless unfair.
    
    \section{Three-player tournaments}\label{sect:threeplayer}

It is easy, though non-trivial, to show that every $2$-player symmetric
    and honest tournament is fair - see Proposition \ref{prop:twoplayer} below.
    Already for three players, this breaks down however. Let $N \geq 2$ and
    consider the following two tournaments:
    \\
    \\
      {\sc Tournament $\bm{T}_1 = \bm{T}_{1, N}$}: The rules are as follows:
      \par \emph{Step 1:} Choose one of the three players uniformly at random.
      Let $i$ denote the chosen player and $j, \, k$ denote the remaining
      players.
      \par
      \emph{Step 2:} Let $j$ and $k$ play $N$ matches.
      \par - If one of them, let's say $j$, wins at least $\frac{3N}{4}$
      matches, then the winner of the tournament is chosen by tossing a fair
      coin between $j$ and $i$.
      \par - Otherwise, the winner of the tournament is chosen by tossing a
      fair coin between $j$ and $k$.
      \\
      \\
      {\sc Tournament $\bm{T}_2 = \bm{T}_{2,\,N}$}: The rules are as follows:
      \par \emph{Step 1:} Choose one of the three players uniformly at random.
      Let $i$ denote the chosen player and $j, \, k$ denote the remaining
      players.
      \par
      \emph{Step 2:} Let $j$ and $k$ play $N$ matches.
      \par - If one of them wins at least $\frac{3N}{4}$ matches, then he is
      declared the winner of the tournament.
      \par - Otherwise, $i$ is declared the winner of the tournament.
      \\
      \par It is easy to see that both $\bm{T}_1$ and $\bm{T}_2$ are symmetric
      and honest (though not strictly honest), for any $N$. Now let
      $p_{12} = p_{23} = \frac{1}{2}$ and $p_{13} = 1$, so that the matrix
      $P = (p_{ij})$ is doubly monotonic, and let's analyze the
      corresponding specializations $\bm{\mathcal{T}}_1$, $\bm{\mathcal{T}}_2$
      of
      each tournament as $N \rightarrow \infty$.
      \par \emph{Case 1:} Player $1$ is chosen in Step 1. In Step 2,
      by the law of large numbers, neither
      $2$ nor $3$ will win at least $\frac{3N}{4}$ matches, asymptotically
      almost surely (a.a.s.). Hence, each of $2$ and $3$ wins
      $\bm{\mathcal{T}}_1$ with probability tending to $\frac{1}{2}$, while
      $1$ a.a.s. wins $\bm{\mathcal{T}}_2$.
      \par \emph{Case 2:} Player $2$ is chosen in Step 1. In Step 2, player
      $1$ will win all $N$ matches. Hence, each of $1$ and $2$ wins
      $\bm{\mathcal{T}}_1$ with probability $\frac{1}{2}$, while $1$ wins
      $\bm{\mathcal{T}}_2$.
      \par
      \emph{Case 3:} Player $3$ is chosen in Step 1. In Step 2, neither $1$
      nor $2$ will win at least $\frac{3N}{4}$ matches, a.a.s.. Hence, each of
      $1$ and $2$ wins $\bm{\mathcal{T}}_1$ with probability tending to
      $\frac{1}{2}$, while $3$ a.a.s. wins $\bm{\mathcal{T}}_2$.
      \par Hence, as $N \rightarrow \infty$, we find that
      \begin{equation}\label{eq:unfairnodes}
        \bm{wv}(\bm{\mathcal{T}_1}) \rightarrow \left( \frac{1}{3}, \, \frac{1}{2}, \, \frac{1}{6} \right) \;\;\; {\hbox{and}} \;\;\; \bm{wv}(\bm{\mathcal{T}_2}) \rightarrow \left( \frac{2}{3}, \, 0, \, \frac{1}{3} \right).
        \end{equation}
      Indeed, we get unfair specializations already for $N = 2$, in which case
      the dichotomy in Step 2 is simply whether or not a player wins both
      matches. One may check that, for $N = 2$,
      \begin{equation*}\label{eq:unfairex}
        \bm{wv}(\bm{\mathcal{T}_1}) = \left( \frac{3}{8}, \, \frac{5}{12}, \, \frac{5}{24} \right) \;\;\; {\hbox{and}} \;\;\; \bm{wv}(\bm{\mathcal{T}_2}) = \left( \frac{7}{12}, \, \frac{1}{6}, \, \frac{1}{4} \right).
      \end{equation*}
      
      We can think of  $\bm{\mathcal{T}}_1$ as trying to give an advantage to
      player $2$ over player $1$, and $\bm{\mathcal{T}}_2$ trying to give an
      advantage to player $3$ over player $2$. It is natural to ask if it is
      possible to improve the tournaments in this regard. Indeed the difference
      in winning probabilities for players $1$ and $2$ in $\bm{\mathcal{T}}_1$
      is only $\frac12-\frac13=\frac16$, and similarly the winning
      probabilities for
      players $3$ and $2$ in $\bm{\mathcal{T}}_2$
      only differ by $\frac13$. In particular, is it
      possible to modify  $\bm{\mathcal{T}}_1$ such that $\pi_1$ goes below
      $\frac13$ or such that $\pi_2$ goes above $\frac12$? Is it possible to
      modify $\bm{\mathcal{T}}_2$ such that $\pi_3$ goes above $\frac13$? The
      answer to both of these questions turns out to be ``no'', as we
      will show below. In fact, these two tournaments are, in a sense, the
      two unique maximally unfair symmetric and honest $3$-player tournaments.

      We begin with two lemmas central to the study of symmetric and honest
      tournaments for an arbitrary number of players.
        \begin{lemma}\label{lem:symmetry}
          Let $\bm{T}$ be a symmetric $n$-player tournament. If $p_{ik} = p_{jk}$
          for all $k = 1,\,\dots,\,n$, then $\pi_i = \pi_j$.
\end{lemma}

\begin{proof}
  Follows immediately from the definition of symmetry by taking $\sigma$
  to be the permutation that swaps $i$ and $j$.
\end{proof}

\begin{lemma}\label{lem:honesty}
  Let $\bm{T}$ be an honest $n$-player tournament and let
  $P=(p_{ij})_{i,j\in[n]} \in\mathcal{M}_n$. Then, for any
  $k \neq l$, $\pi_k=\pi_k(P)$ is increasing in $p_{kl}$.
  \end{lemma}
As the proof of this lemma is a bit technical, we will delay this until the
end of the section.

In applying Lemma \ref{lem:honesty},
it is useful to introduce some terminology. We will use the terms
\emph{buff} and \emph{nerf} to refer to the act of increasing, 
respectively decreasing, one player's match-winning probabilities while leaving
the probabilities between any other pair of players constant{\footnote{These
    terms will be familiar to computer gamers.}}.

\begin{proposition}\label{prop:upperbd}
  Let $n\geq 2$ and let $\bm{T}$ be a symmetric and honest $n$-player
  tournament. For any $P\in\mathcal{D}_n$ and any $i>1$ we have
  $\pi_i(P)\leq\frac12$.
\end{proposition}
\begin{proof}
  Given $P\in\mathcal{D}_n$, we modify this to the matrix $P'$ by buffing player
  $i$ to be equal to player $1$, that is, we put $p'_{i1}=\frac{1}{2}$ and for
  any $j\not\in \{1,\,i\}$, $p'_{ij}=p_{1j}$. By Lemma \ref{lem:honesty},
  $\pi_i(P)\leq \pi_i(P')$. But by Lemma \ref{lem:symmetry},
  $\pi_1(P') = \pi_i(P')$. As the winning probabilities over all players
  should sum to $1$, this means that $\pi_i(P')$ can be at most $\frac{1}{2}$.
\end{proof}

\begin{proposition}\label{prop:twoplayer} Every symmetric and honest
  $2$-player tournament is fair. Moreover, for any $p\in [\frac12, 1]$, there
  is a specialization of an honest and symmetric $2$-player tournament where
  $\pi_1=p$ and $\pi_2=1-p$.
\end{proposition}
\begin{proof}
  By Proposition \ref{prop:upperbd}, any doubly monotonic specialization of such
  tournament satisfies $\pi_2\leq\frac12$ and thereby $\pi_1\geq\pi_2$. On the
  other hand, for any $p \in \left[ \frac{1}{2}, \, 1 \right]$, if $p_{12} = p$
  and the tournament consists of a single match, then $\pi_1 = p$.
\end{proof}

\begin{proposition}\label{prop:threeplayer} Let $\bm{T}$ be a symmetric and
  honest $3$-player tournament. Then, for any $P\in\mathcal{D}_3$,
  $\pi_1\geq\frac13$, $\pi_2\leq\frac12$ and $\pi_3\leq\frac13$.
\end{proposition}
\begin{proof}
The second inequality was already shown in Proposition \ref{prop:upperbd}.

Let us consider the bound for player $1$. Given $P$ we construct a matrix $P'$
by {nerfing} player $1$ such that he becomes identical to player $2$. That is,
we let $p'_{12}=\frac12$ and $p'_{13}=p_{23}$.  This reduces the winning
probability of player $1$, i.e. $\pi_1(P')\leq \pi_1(P)$, and by symmetry
$\pi_1(P')=\pi_2(P')$. We now claim that this common probability for players
$1$ and $2$ is at least $\frac13$. To see this, suppose we construct $P''$ from
$P'$ by {buffing} player $3$ to become identical to players $1$ and $2$, i.e.
$p''_{ij}=\frac12$ for all $i, j$. On the one hand, this increases the winning
probability of player $3$, i.e. $\pi_3(P'')\geq\pi_3(P')$, but on the other
hand, by symmetry we now have $\pi_1(P'')=\pi_2(P'')=\pi_3(P'')=\frac13$.
Hence, $\pi_3(P')\leq\frac13$ and hence $\pi_1(P')=\pi_2(P')\geq\frac13$, as
desired.

The bound for player $3$ can be shown analogously. We first buff player $3$ to
make him identical to player $2$, and then nerf $1$ to become identical to the
other two players.
\end{proof}

For each $n \in \mathbb{N}$, let $\mathcal{P}_n$ denote the convex polytope of
$n$-dimensional probability vectors, i.e.:
\begin{equation*}
  \mathcal{P}_n = \{ (x_1, \, \dots, \, x_n) \in \mathbb{R}^n : x_i \geq 0 \, \forall \, i \; {\hbox{and}} \; \sum_{i=1}^{n} x_i = 1 \}.
\end{equation*}
Let $\mathcal{F}_n \subset \mathcal{P}_n$ be the closed, convex subset
\begin{equation*}
  \mathcal{F}_n = \{ (x_1, \, \dots, \, x_n) \in \mathcal{P}_n : x_1 \geq x_2 \geq \dots \geq x_n \}.
\end{equation*}
We call $\mathcal{F}_n$ the $n$-dimensional \emph{fair} set. 
A vector $\bm{x} = (x_1, \, \dots, \, x_n) \in \mathcal{P}_n$ will be
said to be \emph{achievable} if there is a matrix $P \in \mathcal{D}_n$ and a symmetric, honest $n$-player tournament
$\bm{T}$ such that $\bm{wv}(\bm{T}, P) = \bm{x}$.
We denote by $\mathcal{A}_n$ the closure of the set of achievable vectors in $\mathcal{P}_n$. Note that Proposition \ref{prop:twoplayer} says
that $\mathcal{A}_2 = \mathcal{F}_2$, whereas we already know from
(\ref{eq:unfairnodes}) that $\mathcal{A}_3 \neq \mathcal{F}_3$.

\begin{figure}
\includegraphics[scale=.8]{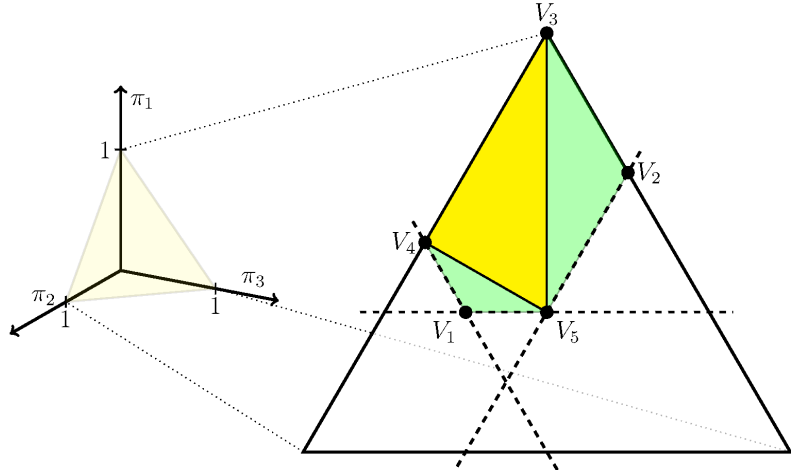}
\caption{\label{fig:polytop_2} Illustration of the set $\mathcal{A}_3$, the closure of the set of achievable win vectors in symetric and honest $3$-player tournaments. The set $\mathcal{P}_3$ is illustrated by the triangle on the right with corners (top), (bottom left), (bottom right) corresponding to the win vectors $(1, 0, 0)$, $(0, 1, 0)$ and $(0, 0, 1)$ respectively. The fair set $\mathcal{F}_3$ is the triangle with corners $V_3=(1, 0, 0), V_4=(\frac12, \frac12, 0)$ and $V_5 =(\frac13, \frac13, \frac13)$. The dotted lines show the three inequalities $\pi_1 \geq \frac13$ (horizontal), $\pi_2\leq \frac12$ (down right diagonal) and $\pi_3 \leq \frac13$ (up right diagonal), as shown in Proposition \ref{prop:threeplayer}. This means that all achievable win vectors are contained in the remaining set, i.e. the convex pentagon with corners $V_3, V_4, V_5$ together with the unfair points $V_1=(\frac13, \frac12, \frac16)$ and $V_2=(\frac23, 0, \frac13)$. We show in Theorem \ref{thm:threeplayer} that every point in this set, except possibly some points on the boundary, is achievable. Thus $\mathcal{A}_3$ is equal to this pentagon.
}\end{figure}

The following result summarizes our findings for symmetric and honest
$3$-player tournaments. This is illustrated in Figure \ref{fig:polytop_2}.
\begin{theorem}\label{thm:threeplayer}
  $\mathcal{A}_{3} = \left\{ (x_1, \, x_2, \, x_3) \in \mathcal{P}_3 : x_1 \geq \frac{1}{3}, \, x_2 \leq \frac{1}{2}, \, x_3 \leq \frac{1}{3} \right\}$. 
\end{theorem}

\begin{proof}
  Denote the above set by $\mathcal{S}$. By Proposition \ref{prop:threeplayer},
  we know that $\mathcal{A}_{3} \subseteq \mathcal{S}$, so it only remains to
  prove that $\mathcal{S} \subseteq \mathcal{A}_{3}$. We start with two
  observations:
  \begin{itemize}
    \item $\mathcal{S}$ is a convex polygon with five vertices:
  \begin{equation*}\label{eq:vertices}
    V_1 = \left( \frac{1}{3}, \, \frac{1}{2}, \, \frac{1}{6} \right), \;\; V_2 = \left( \frac{2}{3}, \, 0, \, \frac{1}{3} \right), \;\; V_3 = (1,\,0,\,0), \;\; V_4 = \left( \frac{1}{2}, \, \frac{1}{2}, \, 0 \right), \;\; V_5 = \left( \frac{1}{3}, \, \frac{1}{3}, \, \frac{1}{3} \right). 
  \end{equation*}
\item Suppose $\bm{\mathcal{T}}^0$, $\bm{\mathcal{T}}^1$ are specializations
  of symmetric and honest $n$-player tournaments $\bm{T}^0$, $\bm{T}^1$
  respectively, and with the same matrix $P \in \mathcal{M}_n$. For
  $p \in [0,\,1]$ we let $\bm{T}^p$ denote the tournament: ``With probability
  $p$ play $\bm{T}^0$ and with probability $1-p$ play $\bm{T}^1$''. Clearly,
  $\bm{T}^p$ is also symmetric and honest for any $p$ and, if
  $\bm{\mathcal{T}}^p$ is its specialization for the matrix $P$, then
  $\bm{wv}(\bm{\mathcal{T}}^p) = p \cdot \bm{wv}(\bm{\mathcal{T}}^0) + (1-p) \cdot \bm{wv}(\bm{\mathcal{T}}^1)$.
    \end{itemize}
  \par It follows from these observations
  that, in order to prove that $\mathcal{S} \subseteq \mathcal{A}_3$, it
  suffices to construct, for each $i = 1,\,\dots,\,5$, a sequence
  $\bm{T}_{i,\,N}$ of symmetric and honest tournaments
  such that $\bm{wv}(\bm{\mathcal{T}}_{i,\,N}) \rightarrow V_i$ as
  $N \rightarrow \infty$, where $\bm{\mathcal{T}}_{i,\,N}$ is the specialization of
  $\bm{T}_{i,\,N}$ by the unique matrix $P = (p_{ij}) \in \mathcal{D}_{3}$
  satisfying $p_{12} = p_{23} = \frac{1}{2}$, $p_{13} = 1$.
  \par Indeed, we've already constructed appropriate sequences for
  $i = 1,\,2$, by (\ref{eq:unfairnodes}), so it remains to take care of
  $i = 3,\,4,\,5$. 
  \\
  \\ 
    {\sc Tournament $\bm{T}_{3,\,N}$}: Play $N$ iterations of round-robin. Choose
    the winner uniformly at random from among the players with the maximum
    number of wins.
    \par It is clear that $\bm{T}_{3,\,N}$ is symmetric and honest and
    that $\bm{wv}(\bm{\mathcal{T}}_{3,\,N}) \rightarrow V_3$ as
    $N \rightarrow \infty$.
    \\
    \\
      {\sc Tournament $\bm{T}_{4,\,N}$}: Play $N$ iterations of round-robin.
      Choose a player uniformly at random from among those with the minimum
      number of wins. Flip a coin to determine the winner among the two
      remaining players.
      \par It is clear that $\bm{T}_{4,\,N}$ is symmetric and honest and that
      $\bm{wv}(\bm{\mathcal{T}}_{4,\,N}) \rightarrow V_4$ as
      $N \rightarrow \infty$. 
      \\
      \\
        {\sc Tournament $\bm{T}_{5}$:} Just choose the winner uniformly at
        random. Obviously $\bm{wv}(\bm{\mathcal{T}}_5) = V_5$ and the
        tournament is symmetric and honest. 
\end{proof}

To conclude this section, we finally give the proof of Lemma \ref{lem:honesty}.
\begin{proof}[Proof of Lemma \ref{lem:honesty}.]

  Fix $k, l\in[n]$ and $\delta>0$. Consider two matrices
  $P=(p_{ij}), P'=(p'_{ij}) \in \mathcal{M}_n$ such that $p'_{kl}=p_{kl}+\delta$,
  $p'_{lk}=p_{lk}-\delta$ and $p'_{ij}=p_{ij}$ whenever $\{i, j\}\neq \{k, l\}$.
  The proof will involve
  interpolating between the specializations $(\bm{T}, P)$ and $(\bm{T}, P')$
  by a sequence of what we'll call ``tournaments-on-steroids''.
  
  For a given $r\geq 0$ we imagine that we play the tournament $\bm{T}$ where,
  in the first $r$ matches, winning probabilities are determined by $P'$, and
  after that according to $P$. The idea is that, at the beginning of the
  tournament, we give player $k$ a performance enhancing drug that only works
  against $l$, and only lasts for the duration of $r$ matches (regardless of
  whether he plays in those matches or not). 
  With some slight abuse of terminology, we will consider these as
  specializations
  of $\bm{T}$, and denote them by $\bm{\mathcal{T}}^r$, and the corresponding
  winning probability of a player $i\in [n]$ by $\pi_i^r$. Clearly
  $\bm{\mathcal{T}}^0 = (\bm{T}, P)$, and taking $m$ equal to the maximum
  number of matches played in $\bm{T}$, it follows that
  $\bm{\mathcal{T}}^m = (\bm{T}, P')$. Hence, it suffices to show that
  $\pi_k^r$ is increasing in $r$.

  Suppose we run the specializations $\bm{\mathcal{T}}^r$ and
  $\bm{\mathcal{T}}^{r+1}$ until either $\bm{T}$ chooses a pair of players to
  meet each other in match $r+1$, or a winner is determined before this happens.
  As both specializations evolve according to the same probability distribution
  up until this point, we may assume that both specializations have behaved
  identically so far. The only way the winning probability for player $k$ can
  differ in the two specializations from this point onwards
  is if match $r+1$ is between
  players $k$ and $l$. Assuming this is the case, let $\pi_k^+$ denote the
  probability that $k$ wins the tournament conditioned on him winning the
  current
  match and assuming all future matches are determined according to $P$, that
  is, according to the specialization $(\bm{T}, P)$. Similarly $\pi_k^-$ denotes
  the probability that he wins conditioned on him losing the match. This means
  that the winning probability for $k$ is
  $p_{kl}\cdot\pi_k^+ + p_{lk}\cdot\pi_k^-$ in $\bm{\mathcal{T}}^{r}$ and
  $p'_{kl}\cdot\pi_k^+ + p'_{lk}\cdot\pi_k^-$ in $\bm{\mathcal{T}}^{r+1}$. But by
  honesty, $\pi_k^+\geq \pi_k^-$, from which it is easy to check that the winning
  probability is at least as high in $\bm{\mathcal{T}}^{r+1}$ as in
  $\bm{\mathcal{T}}^{r}$. We see
  that, for any possibility until match $r+1$ is played, the probability for
  $k$ to win in $\bm{\mathcal{T}}^{r+1}$ is at least as high as in
  $\bm{\mathcal{T}}^r$. Hence $\pi_k^{r+1}\geq \pi_k^r$, as desired.
\end{proof}

\begin{remark}\label{rem:unboundedcouple}

  {\bf (i)} The above proof still works without assuming a bound on the number
  of matches in $\bm{T}$. The only difference will be that $(\bm{T}, P')$ is
  now the limit of $\bm{\mathcal{T}}^r$ as $\rightarrow\infty$.

  {\bf (ii)} If $\bm{T}$ is strictly honest, one can see that
  $\pi_k^{r+1}>\pi_k^r$ for any $P \in \mathcal{M}^{o}_{n}$ and any $r$ such that
  there is a positive probability that match $r+1$ is between players $k$ and
  $l$. Hence, $\pi_k(P)$ is strictly increasing in $p_{kl}$ in this case.
\end{remark}

\section{$n$-Player Tournaments}\label{sect:nplayer}

Already for $n=4$, it appears to be a hard problem to determine which win
vectors are achievable. The aim of this section is to present partial results
in this direction. As we saw in the previous section, $\mathcal{A}_3$ can be completely
characterized by the minimum and maximum win probability each player can
attain. Thus, a natural starting point to analyze $\mathcal{A}_n$ for
$n\geq 4$ is to try to generalize this. For each $i \in [n]$, let
\begin{eqnarray*}
  \Pi^{i,\,n} := \max \{ x_i : (x_1,\,\dots,\,x_n) \in \mathcal{A}_n \}, \\
  \Pi_{i,\,n} := \min \{ x_i : (x_1,\,\dots,\,x_n) \in \mathcal{A}_n \}.
\end{eqnarray*}
In other words, $\Pi^{i,\,n}$ (resp. $\Pi_{i,\,n}$) is the least upper bound
(resp. greatest lower bound) for the win probability for player $i$, taken over
all doubly monotonic specializations of all symmetric and honest $n$-player
tournaments.

It is not too hard to construct a sequence of doubly monotonic specializations of
symmetric and honest tournaments such that $\pi_1\rightarrow 1$. Thus we have
$\Pi^{1,\,n}=1$ and $\Pi_{i,\,n}=0$ for all $i>1$. Moreover, by Proposition
\ref{prop:upperbd}, $\Pi^{i,\,n}\leq \frac12$ for all $i>1$. We can extract a
little more information by using the the technique of
``buffing and nerfing a player'' which was used in Propositions
\ref{prop:upperbd} and \ref{prop:threeplayer}.

\begin{proposition}\label{prop:generalbuffnerf}
  (i) For every $n \in \mathbb{N}$, $\Pi^{i,\,n}$ is a decreasing function
  of $i$.
  \par (ii) $\Pi^{3,\,4} \leq \frac{3}{8}$.
  \par (iii) $\Pi_{1,\,4} \geq \frac{1}{6}$.
\end{proposition}

\begin{proof}
  \emph{(i)} Suppose, on the contrary, that $\Pi^{i+1,\,n} > \Pi^{i,\,n}$, for
  some $n \geq 2$ and $1 \leq i < n$. Then there must exist some symmetric and
  honest $n$-player tournament $\bm{T}$ and some matrix
  $P \in \mathcal{D}_n$ such that $\pi_{i+1}(P) > \Pi^{i,\,n}$. Now buff player
  $i+1$ until he is indistinguishable from $i$ (according to the same
  kind of procedure as in the proof of Proposition \ref{prop:upperbd}).
  Let $P'$ be the resulting matrix. By symmetry and honesty we then have
  $\Pi^{i,\,n} \geq \pi_{i}(P') =
  \pi_{i+1}(P') \geq \pi_{i+1}(P) > \Pi^{i+1,\,n}$, a contradiction.
  \par \emph{(ii)} Let $\bm{T}$ be any symmetric and honest
  $4$-player tournament and let $P \in \mathcal{D}_n$. Perform the following
  three modifications of the
  specialization:
  \par \emph{Step 1:} Buff player $3$ until he is indistinguishable from $2$.
  \par \emph{Step 2:} Nerf player $1$ until he
  is indistinguishable from $2$ and $3$.
  \par \emph{Step 3:} Buff player $4$ until he is indistinguishable from
  $1,\,2$ and $3$.
  \\
  Let $P', P''$ and $P'''$ be the
  corresponding matrices at the end of Steps $1, 2$ and $3$ respectively. By
  Lemmas \ref{lem:symmetry} and
  \ref{lem:honesty}, we first have
  \begin{equation}\label{eq:step1}
    \pi_{3}(P') \geq \pi_3(P), \;\;\;\;\;\; \pi_{2}(P') = \pi_{3}(P').
  \end{equation}
  The latter equality implies, in particular, that
  \begin{equation}\label{eq:step11}
    \pi_{1}(P') \leq 1-2\pi_{3}(P').
  \end{equation}
  A second application of Lemmas \ref{lem:symmetry} and \ref{lem:honesty}
  implies that
  \begin{equation}\label{eq:step2}
    \pi_{1}(P'') \leq \pi_{1}(P'), \;\;\;\;\;\; \pi_{1}(P'')=\pi_{2}(P'')=\pi_{3}(P'').
  \end{equation}
  A third application yields
  \begin{equation}
    \pi_{4}(P''') \geq \pi_4(P''), \;\;\;\;\;\; \pi_{1}(P''')=\pi_{2}(P''')=\pi_{3}(P''')=\pi_{4}(P''') = \frac{1}{4}.
  \end{equation}
  Putting all this together, we have
  \begin{eqnarray*}
    1 = 3\pi_{1}(P'') + \pi_{4}(P'') \leq 3(1-2\pi_{3}(P')) +
    \frac{1}{4} \Rightarrow \pi_{3}(P') \leq \frac{3}{8} \Rightarrow \pi_3(P) \leq \frac{3}{8}.
  \end{eqnarray*}
  \par \emph{(iii)} As before, let $\bm{T}$ be any symmetric and honest
  $4$-player tournament and let $P \in \mathcal{D}_n$. We must show that
  $\pi_1(P) \geq \frac{1}{6}$. Perform the following two modifications of the
  specialization:
  \par \emph{Step 1:} Nerf player $1$ until he is indistinguishable from
  $2$.
  \par \emph{Step 2:} Buff player $3$ until he
  is indistinguishable from $1$ and $2$.
  \\
  Let $P', P''$ be the
  corresponding matrices at the end of Steps $1$ and $2$ respectively. Twice
  applying lemmas \ref{lem:symmetry} and
  \ref{lem:honesty} we get
  \begin{eqnarray}
    \pi_1(P') \leq \pi_1(P), \;\;\;\;\;\; \pi_{1}(P') = \pi_{2}(P'), \label{eq:stepp1} \\
    \pi_3(P'') \geq \pi_{3}(P'), \;\;\;\;\;\; \pi_1(P'') = \pi_2(P'') = \pi_3(P''). \label{eq:stepp2}
  \end{eqnarray}
  From (\ref{eq:stepp2}) we deduce that $\pi_3(P') \leq \frac{1}{3}$. By a
  similar argument, where in Step $2$ one instead buffs $4$ to the level of
  $1$ and $2$, one shows that $\pi_4(P') \leq \frac{1}{3}$. Then, with the
  help of (\ref{eq:stepp1}), we have
  \begin{eqnarray*}
    1 = \pi_1(P') + \pi_2(P') + \pi_3(P') + \pi_4(P') \leq 2\pi_1(P') + 2\cdot \frac{1}{3} \Rightarrow \pi_1(P') \geq \frac{1}{6} \Rightarrow \pi_1(P) \geq \frac{1}{6}.
    \end{eqnarray*}
\end{proof}

We next present a way to construct many symmetric and honest but unfair
tournaments. For each $n \in \mathbb{N}$, let $\mathcal{G}_n$ denote the
family of labelled
digraphs (loops and multiple arcs allowed) on the vertex set
$\{1,\,2,\,\dots,\,n\}$ whose set of arcs satisfies the following conditions:
\par \emph{Rule 1:} There are exactly two arcs going out from each vertex.
\par \emph{Rule 2:} Every arc $(i,\,j)$ satisfies $j \leq i$.
\par \emph{Rule 3:} If $(i, \, j_1)$ and $(i, \, j_2)$ are the two outgoing
arcs from $i$, then $j_1 = j_2 \Rightarrow j_1 = 1$ or $j_1 = i$. In other
words, if the two arcs have the same destination, then either they are both
loops or the destination is vertex $1$.
\\
\par To each digraph $G \in \mathcal{G}_n$ we associate a vector
$v(G) = (v_1,\,\dots,\,v_n) \in \mathcal{P}_n$ according to the rule
\begin{equation}\label{eq:graphvector1}
  v_i = \frac{{\hbox{indeg}}_{G}(i)}{2n}.
\end{equation}
Note that since, by Rule 1, each vertex has outdegree $2$, we can also write
this formula as
\begin{equation}\label{eq:graphvector2}
  v_i = \frac{1}{n} + \frac{{\hbox{indeg}}_{G}(i) - {\hbox{outdeg}}_{G}(i)}{2n}.
\end{equation}
In what follows, each vector $v(G)$ will be interpreted as the win vector of a
certain symmetric and honest tournament. According to (\ref{eq:graphvector2}),
the arcs of $G$ instruct us how to ``redistribute'' win probabilities amongst
the players, starting from the uniform distribution, where each arc
``carries with it'' $\frac{1}{2n}$ of probability.
\\
\par Let $\mathcal{A}^{*}_{n}$ denote the convex hull of all vectors
$v(G)$, $G \in \mathcal{G}_n$. It is easy to see that $\mathcal{A}^{*}_{1}$ is
the single point $(1)$ - the only digraph in $\mathcal{G}_1$ consists of the
single vertex $1$ with two loops. For $n \geq 2$, the number of digraphs in
$\mathcal{G}_n$ is $\prod_{i=2}^{n} 2 + \binom{i}{2}$
since, for each $i \geq 2$, the possibilities for the two outgoing arcs from
vertex $i$ are:
\par - send both to $i$ ($1$ possibility),
\par - send both to $1$ ($1$ possibility),
\par - send them to distinct $j_1, \, j_2 \in \{1,\,\dots,\,i\}$
($\binom{i}{2}$ possibilities).
\\
The number of corners in the convex polytope $\mathcal{A}^{*}_{n}$ is, however,
much less than this. For a digraph $G$ to correspond to a corner of
$\mathcal{A}^{*}_{n}$, there must exist some vector
$\bm{a} = (a_1,\,\dots,\,a_n) \in \mathbb{R}^n$ such that $v(G)$ is the
unique maximizer, in
$v(\mathcal{G}_n)$, of the sum $\sum_{i=1}^{n} a_i v_i(G)$. We can assume that the
coefficients $a_i$ are distinct numbers. For a given
vector $\bm{a}$, a digraph which maximizes the sum is determined by the
following procedure: List the components of $\bm{a}$ in decreasing order, say
$a_{i_1} > a_{i_2} > \dots > a_{i_n}$. Now draw as many arcs as possible first to
$i_1$, then to $i_2$ and so on, all the while respecting Rules 1,2,3 above.
\par We see that the resulting digraph depends only on the ordering of the
components of $\bm{a}$, not on their exact values. In other words, there is a
well-defined map $f: \mathcal{S}_n \rightarrow \mathcal{P}_n$ from permutations
of $\{1,\,\dots,\,n\}$ to corners of $\mathcal{A}^{*}_{n}$,
$f(\sigma) = v(G_{\sigma})$, where, for $\sigma = (\sigma_1,\,\dots,\,\sigma_n) \in \mathcal{S}_n$, the digraph $G_{\sigma}$ is given by the procedure:
\par ``Draw as many arcs as possible first to vertex $\sigma_1$, then to
$\sigma_2$ and so on, all the while respecting Rules 1, 2, 3''.

\begin{table}[ht!]
  \begin{center}
    \begin{tabular}{|c|c|c|} \hline
      $\sigma$ & $G_{\sigma}$ & $v(G_{\sigma})$ \\ \hline \hline
      $(1,\,2)$ & \begin{tikzpicture}[->,>=stealth',shorten >=1pt,auto,node distance=4cm,
                thick,main node/.style={circle,draw,font=\Large\bfseries}]

  \node[main node] (1) {1};
  \node[main node] (2) [right of=1] {2};

  \path
  (1) edge [loop above] node {} (1)
        edge [loop below] node {} (1)
        (2) edge [bend right] node {} (1)
        edge [bend left] node {} (1);  
\end{tikzpicture} & $(1,\,0)$ \\ \hline
      $(2,\,1)$ & \begin{tikzpicture}[->,>=stealth',shorten >=1pt,auto,node distance=4cm,
                thick,main node/.style={circle,draw,font=\Large\bfseries}]

  \node[main node] (1) {1};
  \node[main node] (2) [right of=1] {2};

  \path
    (1) edge [loop above] node {} (1)
        edge [loop below] node {} (1)
    (2) edge [loop above] node {} (2)
        edge [loop below] node {} (2);      
\end{tikzpicture} & $\left( \frac{1}{2}, \, \frac{1}{2} \right)$ \\ \hline \hline
      $(1,\,2,\,3) \; {\hbox{or}} \; (1,\,3,\,2)$ & \begin{tikzpicture}[->,>=stealth',shorten >=1pt,auto,node distance=2cm,
                thick,main node/.style={circle,draw,font=\Large\bfseries}]

  \node[main node] (1) {1};
  \node[main node] (2) [below left of=1] {2};
  \node[main node] (3) [below right of=1] {3};

  \draw
    (1) to [out=150, in=120, looseness=8] (1);
  \draw
  (1) to [out=30, in=60, looseness=8] (1);
        \path
        (2) edge [bend left] node {} (1)
            edge [bend right] node {} (1) 
    (3) edge [bend left] node {} (1)
        edge [bend right] node {} (1);      
\end{tikzpicture} & $(1,\,0,\,0)$ \\ \hline
      $(2,\,1,\,3)$ & \begin{tikzpicture}[->,>=stealth',shorten >=1pt,auto,node distance=2cm,
                thick,main node/.style={circle,draw,font=\Large\bfseries}]

  \node[main node] (1) {1};
  \node[main node] (2) [below left of=1] {2};
  \node[main node] (3) [below right of=1] {3};

   \draw
    (1) to [out=150, in=120, looseness=8] (1);
  \draw
  (1) to [out=30, in=60, looseness=8] (1);
   \draw
    (2) to [out=120, in=150, looseness=8] (2);
  \draw
  (2) to [out=240, in=210, looseness=8] (2);
  \path
    (3) edge node {} (1)
        edge node {} (2);      
\end{tikzpicture} & $\left( \frac{1}{2}, \frac{1}{2}, \, 0 \right)$ \\ \hline
      $(2,\,3,\,1)$ & \begin{tikzpicture}[->,>=stealth',shorten >=1pt,auto,node distance=2cm,
                thick,main node/.style={circle,draw,font=\Large\bfseries}]

  \node[main node] (1) {1};
  \node[main node] (2) [below left of=1] {2};
  \node[main node] (3) [below right of=1] {3};

 \draw
    (1) to [out=150, in=120, looseness=8] (1);
  \draw
  (1) to [out=30, in=60, looseness=8] (1);
   \draw
    (2) to [out=120, in=150, looseness=8] (2);
  \draw
  (2) to [out=240, in=210, looseness=8] (2);
  \path
  (3) edge [loop above] node {} (3)
  edge node {} (2);    
\end{tikzpicture} & $\left( \frac{1}{3}, \, \frac{1}{2}, \, \frac{1}{6} \right)$ \\ \hline
      $(3,\,1,\,2)$ & \begin{tikzpicture}[->,>=stealth',shorten >=1pt,auto,node distance=2cm,
                thick,main node/.style={circle,draw,font=\Large\bfseries}]

  \node[main node] (1) {1};
  \node[main node] (2) [below left of=1] {2};
  \node[main node] (3) [below right of=1] {3};

\draw
    (1) to [out=150, in=120, looseness=8] (1);
  \draw
  (1) to [out=30, in=60, looseness=8] (1);
   \draw
    (3) to [out=60, in=30, looseness=8] (3);
  \draw
  (3) to [out=300, in=330, looseness=8] (3);
  \path
  (2) edge [bend left] node {} (1)
  edge [bend right] node {} (1);  
\end{tikzpicture} & $\left( \frac{2}{3}, \, 0, \, \frac{1}{3} \right)$ \\ \hline
      $(3,\,2,\,1)$ & \begin{tikzpicture}[->,>=stealth',shorten >=1pt,auto,node distance=2cm,
                thick,main node/.style={circle,draw,font=\Large\bfseries}]

  \node[main node] (1) {1};
  \node[main node] (2) [below left of=1] {2};
  \node[main node] (3) [below right of=1] {3};

  \draw
    (1) to [out=150, in=120, looseness=8] (1);
  \draw
  (1) to [out=30, in=60, looseness=8] (1);
   \draw
    (2) to [out=120, in=150, looseness=8] (2);
  \draw
  (2) to [out=240, in=210, looseness=8] (2);
   \draw
    (3) to [out=60, in=30, looseness=8] (3);
  \draw
  (3) to [out=300, in=330, looseness=8] (3); 
  \end{tikzpicture} & $\left( \frac{1}{3}, \, \frac{1}{3}, \, \frac{1}{3} \right)$ \\ \hline
    \end{tabular}
  \end{center}
  \vspace{0.3cm}
  \caption{All $\sigma \in \mathcal{S}_n$, $G_{\sigma} \in \mathcal{G}_n$ and
    corners $v(G_{\sigma})$ of $\mathcal{A}^{*}_{n}$, for $n = 2, \, 3$.}
  \label{tab:corners}
  \end{table}

Table \ref{tab:corners} shows how this works for $n=2$ and $n=3$. The map
$f$ is not injective for any $n \geq 3$ and the exact number of corners in
$\mathcal{A}^{*}_n$ is computed in Proposition \ref{prop:corners} below. For the
time being, the crucial takeaway from Table \ref{tab:corners} is that
$\mathcal{A}^{*}_{2} = \mathcal{A}_2$ and $\mathcal{A}^{*}_3 = \mathcal{A}_3$.
Recall also that $\mathcal{A}^{*}_1 = \mathcal{A}_1 = \{ (1)\}$.
\par We are ready to formulate

\begin{conjecture}\label{conj:nplayer}
  $\mathcal{A}^{*}_n = \mathcal{A}_n$, for every $n \in \mathbb{N}$.
\end{conjecture}

Our main result in this section is

\begin{theorem}\label{thm:nplayer}
  $\mathcal{A}^{*}_n \subseteq \mathcal{A}_n$, for every $n \in \mathbb{N}$.   
\end{theorem}

\begin{proof}
  We've already observed that $\mathcal{A}^{*}_{n} = \mathcal{A}_n$ for
  $n=1, 2, 3$. We divide the remainder of the proof into two cases.
  \\
  \\
    {\sc Case I:} $n \geq 5$. Since we can form a
    ``convex combination of tournaments'' - see the proof of Theorem
    \ref{thm:threeplayer} - it suffices to find, for any fixed
    $P \in \mathcal{D}_n$ and for each $G \in \mathcal{G}_n$, a
    sequence $\bm{T}_{G, \, N}$ of symmetric and honest tournaments such that
    $\bm{wv}((\bm{T}_{G,\,N}, \, P)) \rightarrow v(G)$ as
    $N \rightarrow \infty$.
    \par Let $P = (p_{ij})$ be any doubly monotonic matrix such that
    $p_{ij}\neq p_{kl}$ unless either $i=k, \, j=l$ or $i=j, \, k=l$.
    The matrix $P$ is
    henceforth fixed. Let
    \begin{equation}\label{eq:eps}
      \varepsilon_1 := \min_{i \neq j} {\hbox{$| p_{ij} - \frac{1}{2}|$}},
      \;\;\; \varepsilon_2 := \min_{\stackrel{i \neq j, \, k \neq l,}{\{i,\,j\} \neq \{k,\,l\}}} |p_{ij} - p_{kl}|, \;\;\;
      \varepsilon := \frac{1}{2} \min \{ \varepsilon_1, \, \varepsilon_2 \}.
    \end{equation} 
    In other words, $\varepsilon$ is half the minimum difference between two
    distinct numbers appearing in the matrix $P$.
    \par For $N \in \mathbb{N}$ and $G \in \mathcal{G}_n$, the rules of the
    tournament $\bm{T}_{G,\,N}$ are as follows. We remark that the matrix $P$
    here is a fixed parameter as part of the rules and does not (necessarily)
    have anything to do with the specialization. In due course we will, however,
    also have reason to consider the specialization $(\bm{T}_{G, \, N}, \, P)$.
    \\
    \\
    \emph{Step 1:} Present the matrix $P$ to each of the players.
    \\
    \\
    \emph{Step 2:} Choose one of the players uniformly at random. This player
    takes no further part in the tournament. 
    \\
    \\
    \emph{Step 3:} The remaining $n-1$ players play $N$ iterations of
    round-robin.
    
    Once all the matches are finished, each remaining player performs a sequence
    of tasks{\footnote{One can instead imagine that there is a ``referee'' who performs all these tasks, since they are part of the rules for the tournament. We think it's intuitively easier to understand the idea, however, in terms of each player perfoming his own calculations. Note that Step 1 can be removed from the description of the rules if we formulate them in terms of a central referee.}}
    which is a little technical to describe. Informally, he tries to establish
    the identities of the other $n-2$ remainers, as elements from $[n]$, by
    checking the results of all the matches not involving himself and comparing
    with the given matrix $P$. More formally, he does the following:
    \par (a) He makes an arbitrary list
    $(t_1,\,t_2,\,\dots,\,t_{n-2})$ of the other $n-2$
    remainers and computes the elements $q_{ij}$ of an $(n-2) \times (n-2)$
    matrix such that $q_{ij}$ is the fraction of the matches between $t_i$ and
    $t_j$ which were won by $t_i$.
    \par (b) He tries to find a subset $\{u_1,\,\dots,\,u_{n-2}\} \subset [n]$
    such that, for all $1 \leq i < j \leq n-2$,
    \begin{equation}\label{eq:success}
      |q_{ij} - p_{u_i, \, u_j}| < \varepsilon.
    \end{equation}
    Note that, by (\ref{eq:eps}), he can find at most one such
    $(n-2) \times (n-2)$ submatrix of $P$. If he does so, we say that he
    \emph{succeeds} in Step 3.
    \\
    \\
    \emph{Step 4:} For each player that succeeds in Step 3, do the following:
    \par (a) Let $i < j \in [n]$ be the numbers of the
    two rows and columns in $P$ which are excluded from the submatrix he
    identified in Step 3.
    \par (b) For each $l \in [n] \backslash \{i,\,j\}$, compute the fraction
    $r_l$ of matches which he won against the player whom he identified in
    Step 3 with row $l$ of the matrix $P$.
    \par (c) If $r_l > p_{il} - \varepsilon$ for every $l$, then assign this
    player a ``token'' of weight $\frac{n_{ji}}{2}$, where $n_{ji}$ is the number
    of arcs from $j$ to $i$ in the digraph $G$.
    \\
    \\
    \emph{Step 5:} Assign to the player eliminated in Step 2 a token of weight
    $1 - s$, where $s$ is the sum of the weights of the tokens distributed in
    Step 4. The winner of the tournament is now chosen at random, weighted in
    accordance with the distribution of tokens.
    \\
    \par What needs to be proven now is that the tournament $\bm{T}_{G,\,N}$ is
    always well-defined, that is, it can never happen that the total weight of
    the tokens distributed in Step 4 exceeds one. Supposing for the moment that
    this is so, it is clear that the tournament is symmetric and honest, and it
    is also easy to see that
    $\bm{wv}((\bm{T}_{G,\,N}, \, P)) \rightarrow v(G)$ as
    $N \rightarrow \infty$. For if the relative strengths of the $n$ players
    are, \emph{in fact}, given by the matrix $P$ then, as
    $N \rightarrow \infty$, with high probability everyone not eliminated in
    Step 2 will succeed with identifying an $(n-2) \times (n-2)$ submatrix of
    $P$ in Step 3, namely the submatrix corresponding to the \emph{actual}
    rankings of these $n-2$ remainers, and will then have performed well enough
    to be assigned a token in Step 4(c) if and only if their actual ranking is
    higher than that of the player eliminated in Step 2 (note that the weight
    of the token they are assigned will still be zero if there is no
    corresponding arc in the digraph $G$).
    \\
    \par So it remains to prove that the total weight of all tokens assigned in
    Step 4(c) can never exceed one. If at most one player is assigned a token
    of non-zero weight then we're fine, because of Rule 1 in the definition of
    the family $\mathcal{G}_n$. Suppose at least two players are assigned
    tokens of non-zero weight. Let $A,\,B,\,C,\,D,\,\dots$ denote all the
    players not eliminated in Step 2 (these are just letters, not
    numbers) and suppose $A$ and $B$ are assigned non-zero-weight tokens. Since
    each of $A$ and $B$ can see the results of all matches involving
    $C,\,D,\,\dots$, they will identify these with the same $n-3$ elements of
    $[n]$ in Step 3. Note that here we have used the fact that $n \geq 5$. Let
    $\mathcal{S} \subset [n]$ be this $(n-3)$-element subset. This
    leaves three indices $i < j < k \in [n] \backslash \mathcal{S}$.
    We have four options to consider:
    \\
    \\
    \emph{Option 1:} At least one of $A$ and $B$ identifies the other as $k$.
    We show this can't happen. Suppose $A$ identifies $B$ as $k$. Then $B$
    must have performed at about the level expected of $k$ against each of
    $C,\,D,\,\dots$. More precisely, for any 
    $l \in \mathcal{S}$,
    \begin{equation}\label{eq:option1a}
      |r^{B}_{l} - p_{kl}| < \varepsilon.
    \end{equation}
    On the other hand,
    the rules of Step 4 imply that, for $B$ to receive a token, he must have
    performed at least at the level expected of $j$ against each of
    $C,\,D,\,\dots$ (and, indeed, at the level expected of $i$ in the case that
    he failed to identify $A$ as $i$). Precisely, for each $l \in \mathcal{S}$, 
    \begin{equation}\label{eq:option1b}
      r^{B}_{l} > p_{jl} - \varepsilon.
    \end{equation}
    But (\ref{eq:option1a}) and (\ref{eq:option1b}) contradict (\ref{eq:eps}).
    \\
    \\
    \emph{Option 2:} $A$ and $B$ identify one another as $j$. We show that this
    can't happen either. Suppose otherwise. Since $A$ gets a token, it must
    pass the test $r^{A}_{j} > p_{ij} - \varepsilon$. Similarly
    $r^{B}_{j} > p_{ij} - \varepsilon$. But $r^{A}_{j} + r^{B}_{j} = 1$, since
    each of $A$ and $B$ is here computing the fraction of matches it won
    against the other. This implies that
    $p_{ij} < \frac{1}{2} + \varepsilon$, which contradicts (\ref{eq:eps}).
    \\
    \\
    \emph{Option 3:} Each of $A$ and $B$ identifies the other as $i$. Then the
    weight of the token assigned to each is $\frac{n_{kj}}{2}$. But $j > 1$ so
    $n_{kj} \leq 1$, by Rule 3 for the family $\mathcal{G}_n$. Hence it suffices
    to prove that no other player receives a token. Suppose $C$ receives a
    token. $C$ sees the results of matches involving either $A$ or $B$ and any
    of $D,\,\dots$. Since $A$ and $B$ have already identified one another as
    $i$, then $C$ must make the same identification for each, by (\ref{eq:eps}).
    In other words, $C$ cannot distinguish $A$ from $B$, a contradiction.
    \\
    \\
    \emph{Option 4:} $A$ and $B$ identify one another as $i$ and $j$, in some
    order. Since both get non-zero-weight tokens, there must, by Rules 1-3, be
    exactly one arc in $G$ from $k$ to each of $i$ and $j$. So the sum of the
    weights assigned to $A$ and $B$ equals one, and there is no arc in $G$ from
    $k$ to any vertex other than $i$ and $j$. It now suffices to show that no
    other player $C$ receives a positive weight token. The only way $C$ can
    succeed in Step 3 is if it also identifies $A$ and $B$ as $i$ and $j$, and
    if there is some $l \neq k$ such that it identifies
    $\{C, \, Z\} = \{k, \, l\}$, where $Z$ is the player eliminated in Step 2.
    Both $A$ and $B$ must in turn have identified $C$ as $l$. If $k < l$ this
    means that $C$ cannot have played sufficiently well to obtain a token in
    Step 4(c). If $l < k$ then even if $C$ gets a token it will have weight
    zero, since there is no arc in $G$ from $k$ to $l$.
    \\
    \\
      {\sc Case II:} $n=4$. We use the same tournaments $\bm{T}_{G,\,N}$ as in
      {\sc Case I}, but in order to ensure their well-definedness we require,
      in addition to (\ref{eq:eps}), the following conditions on the
      $4 \times 4$ doubly monotonic matrix $P = (p_{ij})$:
      \begin{equation}\label{eq:eps4}
        p_{14} > p_{24} > p_{34} > p_{13} > p_{12} > p_{23}.
      \end{equation}
      Intuitively, player $4$ is useless, while the gap between $1$ and $2$ is
      greater than that between $2$ and $3$. To prove well-definedness, it
      suffices to establish the following two claims:
      \\
      \\
      \emph{Claim 1:} If some player receives a token of weight one, then no
      other player receives a token of positive weight.
      \\
      \\
      \emph{Claim 2:} It is impossible for three players to receive positive
      weight tokens.
      \\
      \par Let $D$ denote the player eliminated in Step 2 and $A,\,B,\,C$ the
      three remainers.
      \\
      \\
      \emph{Proof of Claim 1.} Suppose $A$ receives a token of weight one. The
      rules for $\mathcal{G}_n$ imply that $A$ must identify himself as $1$ and
      there are two arcs in $G$ from $j$ to $1$, where $j$ is the identity
      which $A$ assigns to $D$. We consider two cases.
      \par Case (a): $j = 4$. Suppose, by way of contradiction, that $B$ also
      receives a positive weight token. In order to obtain a token at all, $B$
      cannot have identified himself as $1$, because he has lost more than half
      his matches against $A$. Hence there is no arc in $G$ from $4$ to
      whomever $B$ identifies himself as, so $B$ cannot have identified $D$ as
      $4$. Since $A$ also beat $C$, it must be the case that $B$ identifies
      $C = 4$, $A=1$, $B=2$, $D=3$. But for $B$ to receive a token, he must
      then have won at least $p_{24} - \varepsilon$ of his matches against $C$.
      This contradicts $A$:s identification $\{B, \, C\} = \{2, \, 3\}$, since
      the latter
      would mean that the fraction of matches $B$ won against $C$ was at most
      $p_{23} + \varepsilon$.
      \par Case (b): $j \in \{2,\,3\}$. $A$ must have identified some remainer
      as $4$, say $C$, and then won at least a fraction $p_{14} - \varepsilon$
      of their matches. $C$:s performance against $A$ is so bad that he cannot
      possibly receive a token. Moreover, $B$ observes this and hence must also
      identify $A = 1$, $C = 4$. So if $B$ receives a token, he will have
      agreed with $A$ on the identities of all four players. But then his token
      cannot have positive weight, since there are no more arcs emanating from
      $j$.
      \\
      \\
      \emph{Proof of Claim 2.} Suppose each of $A, \, B, \, C$ receives a
      token. Since $p_{34} > p_{13}$ by (\ref{eq:eps4}), each must identify
      $D = 4$. This is because if anyone has identified you as $4$, you are
        so bad that you can never satisfy the condition to get a token. We
      consider three cases.
      \par Case (a): Someone, say $A$, identifies themselves as $1$. Then,
      without loss of generality, they identify $B = 2$, $C = 3$. Since $A$
      gets a positive weight token, he must at least have won more than half
      of his matches against both $B$ and $C$. Hence, neither $B$ nor $C$ can
      self-identify as $1$ and get a token. Since there are at most two arcs
      emerging from $4$, $B$ and $C$ must identify themselves as the same
      number, one of $2$ and $3$. But $B$ observes the matches between $A$ and
      $C$ and, since $A$ got a token, he won at least a fraction
      $p_{13} - \varepsilon$ of these. Thus $B$ must self-identify as $2$, hence
      so does $C$. But $C$ observes the matches between $A$ and $B$, of which
      $A$ won a majority, hence $C$ must identify $A$ as $1$. But then $C$
      cannot get a token, since he lost at least a fraction
      $p_{13} - \varepsilon$ of his matches against $A$.
      \par Case (b): Nobody self-identifies as $1$, and someone lost at least
      half of their matches against each of the other two. WLOG, let $A$ be
      this ``loser''. The only way $A$ can get a token is if he self-identifies
      as $3$. WLOG, he identifies $B = 1$, $C = 2$. To get a token he must have
      won at least a fraction $p_{32} - \varepsilon$ of his matches against $C$.
      But $C$ beat $A$ and $B$ didn't self-identify as $1$, hence $B$ must have
      identified $C=1$, which means that $C$ won at least a fraction
      $p_{12} - \varepsilon$ of his matches against $A$. This contradicts
      (\ref{eq:eps}), given the additional assumption that
      $p_{12} > p_{23}$ in (\ref{eq:eps4}).
      \par Case (c): Nobody self-identifies as $1$, and everyone beat someone
      else. Without loss of generality, $A$ beat $B$, who beat $C$ who beat
      $A$. First suppose someone, say $A$, self-identifies as $2$. Then he
      must identify $B = 1$, $C=3$. But then he would have to have beaten $C$
      to get a token, a contradiction.
      \par So, finally, we have the possibility that each of
      $A,\,B,\,C$ self-identifies as $3$. Thus each identifies the other two as
      $1$ and $2$, which means that in each pairwise contest, the fraction of
      matches won by the winner lies in the interval
      $(p_{12} - \varepsilon, \, p_{12} + \varepsilon)$.
      Let $r_{CA}$ denote the fraction of matches won by $C$ against $A$.
      Since $C$ beat $A$, the previous analysis implies that 
      $r_{CA} > p_{12} - \varepsilon$. But
      $A$ identifies himself as $3$ and $B$ beat $C$, so he
      must identify $C$ as $2$. Since $A$ gets a token, we must have
      $r_{CA} < p_{23} + \varepsilon$. But these two inequalites for
      $r_{CA}$ contradict (\ref{eq:eps4}) and (\ref{eq:eps}).
      
\end{proof}

\begin{corollary}\label{cor:unfairn}
  $\mathcal{F}_n$ is a proper subset of $\mathcal{A}_n$, for all $n \geq 3$.
\end{corollary}

\begin{proof}
  It is easy to see that $\mathcal{F}_n$ is a proper subset of
  $\mathcal{A}^{*}_{n}$, for each $n \geq 3$. Then apply Theorem
  \ref{thm:nplayer}.
\end{proof}

Given the preceding results, we now return to the consideration of the
  maximum and minimum winning probabilities, $\Pi^{i,\,n}$ and $\Pi_{i,\,n}$
  respectively, attainable by each player $i$. If we want to minimize the first
coordinate in a vector $v(G)$, there should be no arc pointing to $1$ from any
$j> 1$, and just the two loops from $1$ to itself. In that case,
$v_1 (G) = \frac{1}{n}$. For $i \geq 2$, in order to maximize the $i$:th
coordinate of $v(G)$, it is clear that the digraph $G$ should
\par - have one arc from $j$ to $i$, for each $j = i+1,\,\dots,\,n$,
\par - have two loops $(i,\,i)$,
\par - hence, have no arc from $i$ to $k$, for any $k < i$.
\\
For such $G$ we'll have $v_i (G) = \frac{{\hbox{indeg}}_{i}(G)}{2n} =
\frac{n-i+2}{2n} = \frac{1}{2} - \frac{i-2}{2n}$.
Hence, by Theorem
\ref{thm:nplayer}, we have
\begin{equation}
  \Pi_{1,\,n} \leq \frac{1}{n}; \;\;\;\;\;\; \Pi^{i,\,n} \geq \frac{1}{2} - \frac{i-2}{2n}, \; i = 2, \dots, n.
\end{equation}
If Conjecture \ref{conj:nplayer} were true, we'd have equality everywhere.
Note that, by Proposition \ref{prop:upperbd}, we do indeed have the equality
$\Pi^{2,\,n} = \frac{1}{2}$, and by Proposition \ref{prop:generalbuffnerf},
$\Pi^{3, 4} = \frac38$. Other than this, we can't prove
a single outstanding equality for any $n \geq 4$. In particular, for every
$n \geq 4$ it remains open whether $\Pi_{1,\,n} = \Pi^{n,\,n} = \frac{1}{n}$.

Next, we determine the exact number of corners in
$\mathcal{A}^{*}_{n}$:

\begin{proposition}\label{prop:corners}
  There are $\frac{3^{n-1} + 1}{2}$ corners in the convex
  polytope $\mathcal{A}^{*}_{n}$.
\end{proposition}
\begin{proof}
  We must determine the number of elements in the range of the
  function $f : \mathcal{S}_n \rightarrow \mathcal{P}_n$ defined earlier.
  We begin by noting that, in the encoding $f(\sigma)=v(G_\sigma)$,
  we may not need to know
  the entire permutation $\sigma$ in order to construct $G_\sigma$. In
  particular, it suffices to know the subsequence $\sigma'$ of all vertices
  that get assigned incoming arcs. We note that a vertex $i$ has no incoming
  arcs in $G_\sigma$ if and only if it is either preceded by two
  lower-numbered vertices or preceded by the vertex $1$. Therefore, any such
  subsequence
  $\sigma'$ is a sequence of distinct elements in $[n]$ that $(i)$ ends with a
  $1$ and $(ii)$ for any $i$, at most one of
  $\sigma'_1, \sigma'_2, \dots, \sigma'_{i-1}$ is smaller than $\sigma'_i$.
  Conversely, any sequence $\sigma'$ that satisfies $(i)$ and $(ii)$ can
  be extended to a permutation $\sigma$, without affecting which vertices get
  incoming arcs, by putting the missing numbers after
  the '$1$'. Hence the
  possible subsequences $\sigma'$ are characterized by $(i)$ and $(ii)$.

  We claim that the map
  $\sigma'\mapsto v(G_{\sigma'})$ is injective.
  Let $\sigma'$, $\sigma''$ be two distinct such sequences and pick $k$ such
  that $\sigma'_1 = \sigma''_1, \dots, \sigma'_{k-1} = \sigma''_{k-1}$
  and $\sigma'_{k} \neq \sigma''_{k}$, say $\sigma'_{k} < \sigma''_{k}$. To
  prove injectivity it suffices, by (\ref{eq:graphvector1}),
  to show that the vertex $\sigma''_{k}$ 
  has higher indegree in $G_{\sigma''}$
  than in $G_{\sigma'}$. We consider two cases:
  \par Case 1: $\sigma''_{k}$ does not appear at all in the subsequence
  $\sigma'$. Then, simply by how these subsequences were defined,
  $\sigma''_{k}$ has indegree zero in $G_{\sigma'}$ and strictly positive
  indegree in $G_{\sigma''}$.
  \par Case 2: $\sigma''_{k} = \sigma'_{l}$ for some $l > k$. Since
  $\sigma'_k < \sigma''_k$, property
  $(ii)$ applied to $\sigma'$ implies that
  $\sigma'_j = \sigma''_j > \sigma''_k$ for
  every $j = 1,\,\dots,\,k-1$. Hence, in $G_{\sigma''}$, the vertex $\sigma''_k$
  will retain both of its loops, whereas in $G_{\sigma'}$ there will be one arc
  from $\sigma''_k$ to $\sigma'_k$. Moreover, since $\sigma'$ and
  $\sigma''$ agree before the appearance of $\sigma''_k$, which then appears
  first in $\sigma''$, if $v \in [n]$ is any vertex that sends an arc to
  $\sigma''_k$ in $G_{\sigma'}$, then it will send at least as many arcs
  to $\sigma''_k$ in $G_{\sigma''}$. Hence, the total indegree of $\sigma''_k$ will
  be strictly higher in $G_{\sigma''}$ than in $G_{\sigma'}$, as desired. 

  It remains to count the number of sequences $\sigma'$ that satisfy properties
  $(i)$ and $(ii)$. Denote this by $a_n$. Given such a sequence of elements in
  $[n-1]$, we construct a sequence in $[n]$ by either $(1)$ doing nothing, $(2)$
  placing $n$ first in the sequence, or $(3)$ inserting $n$ between the first
  and second element - this is possible for all sequences except the one just
  consisting of a '$1$'. Thus for any $n\geq 2$, we have $a_n = 3a_{n-1}-1$. It
  is easy to check that $a_1=1$ and thus it follows by induction that
  $a_n= \frac{3^{n-1} + 1}{2}$ as desired.
\end{proof}

We close this section by posing a natural question which arises from the
previous discussion, but which remains unknown to us:

\begin{question}\label{quest:boundary}
  For each $n \geq 3$, which boundary points of $\mathcal{A}^{*}_{n}$ are
  achievable ?
\end{question}

  \section{Frugal tournaments}\label{sect:frugal}

  A central idea of the unfair tournaments presented in Sections
  \ref{sect:threeplayer} and \ref{sect:nplayer} is to first
  choose one player uniformly at random to exclude from participation. This
  player won't take part in any matches, though he might still win the
  tournament. Let us call a
  tournament with this property \emph{frugal}, as the organizers won't have
  to pay the attendance costs for one of the players. In the proof of Theorem
  \ref{thm:nplayer}, we constructed symmetric, honest and frugal tournaments
  whose win vector can attain any interior point in $\mathcal{A}_n^*$ for any
  $n\geq 4$. We will now show that,
  under the restriction that the tournament is frugal, nothing outside of
  $\mathcal{A}^{*}_{n}$ can be achieved.
  
\begin{theorem}\label{thm:frugal}
  Let $\bm{T}$ be a symmetric, honest and frugal $n$-player tournament for any
  $n\geq 2$. Then for any $P\in\mathcal{D}_n$,
  $\bm{wv}((\bm{T},\,P))\in\mathcal{A}_n^*$.
\end{theorem}
\begin{corollary} The closure of the set of
  all achievable win vectors for all symmetric,
  honest and frugal $n$-player tournaments equals $\mathcal{A}_n^*$. 
\end{corollary}
\begin{proof} This follows immediately from Theorems \ref{thm:frugal} and
  \ref{thm:nplayer}.
\end{proof}

In order to prove Theorem \ref{thm:frugal}, we need a new formulation of
$\mathcal{A}_n^*$. We say that a matrix $M\in\mathbb{R}^{n\times n}$ is a
\emph{fractional arc flow} if
\begin{align}
m_{ij} \geq 0&\text{ for all }i \geq j,\label{eq:af1}\\
m_{ij} = 0&\text{ for all }i<j,\label{eq:af2}\\
m_{ij} \leq \frac{1}{2}&\text{ for all }j\neq 1, i,\label{eq:af3}\\
\sum_{j=1}^n m_{ij} = 1&\text{ for all }i\in [n].\label{eq:af4}
\end{align}
\begin{lemma}\label{lem:faf}
  For any fractional arc flow $M$, define $v(M)\in\mathbb{R}^n$ by
  $v_j(M)=\frac1n \sum_{i=1}^n m_{ij}$. Then $v(M)\in \mathcal{A}_n^*$.
\end{lemma}
\begin{proof}
  Let $A$ be the set of vectors $v$ that can be obtained from fractional arc
  flows in this way. Clearly, $A$ is a convex
  polytope in $\mathbb{R}^n$. Thus it is
  uniquely defined by the values of $\max_{v\in A} u\cdot v$ for all
  $u \in \mathbb{R}^n$. For a given $u\in\mathbb{R}^n$, it is easy to optimize
  the
  corresponding fractional arc flow. Namely, initially all vertices are given
  a flow of $1$. Go through the indices $j\in [n]$ in the order of decreasing
  $u_j$, with ties broken arbitrarily, and try to send as much remaining flow as
  possible from all $i\geq j$ to $j$. By (\ref{eq:af3}),
  we see that any such optimal $v$ is
  given by $v(G)$ for some $G\in \mathcal{G}_n$. From the
  discussion in the paragraph preceding Conjecture \ref{conj:nplayer}, it is
  easy to see that the
  vector $v(G)$ is also the optimal vector in the maximization problem
  $\max_{v\in\mathcal{A}_n^*} u\cdot v$. Hence $A=\mathcal{A}_n^*$ as desired.
\end{proof}

\begin{proof}[Proof of Theorem \ref{thm:frugal}]
  For any $i\neq j$, let $\bm{T}^i$ denote the modified version of this
  tournament that always excludes player $i$. By possibly precomposing
  $\bm{T}$ with a random permutation of the players, we may assume that the
  rules of $\bm{T}^i$ do not depend on $(a)$ which player $i$ was excluded, and
  $(b)$ the order of the remaining players $[n]\setminus \{i\}$.

  Let $\pi_j^i(P)$ denote the winning probability for player $j$ in the
  specialization $(\bm{T}^i, P)$. Then
  $\pi_j(P) = \frac{1}{n}\sum_{i=1}^n\pi_j^i(P)$. As $\bm{T}$ is honest, it
  follows directly from the definition of honesty that also $\bm{T}^i$ is
  honest, hence $\pi_j^i(P)$ is increasing in $p_{jk}$ for any $k\neq j$.
  Moreover, if two players $i$ and $j$ are identical for a given
  $P\in\mathcal{M}_n$ in the sense that $p_{ik}=p_{jk}$ for all $k\in[n]$, then
  by $(a)$ by $(b)$, $$\pi^i_j(P)=\pi^j_i(P)$$ and
  $$\pi^k_i(P)=\pi^k_j(P)\text{ for any }k\neq i, j.$$
  Using the same argument as in Proposition \ref{prop:upperbd} it follows that,
  for any $P \in \mathcal{D}_n$, 
  $$\pi^i_j(P)\leq \frac{1}{2}\text{ unless either (i) $j=1$, (ii) $i=j$, or (iii) $i=1$ and $j=2$.}$$
  Moreover, for any $P\in\mathcal{D}_n$ and $i<j$, let $P'$ be the matrix
  obtained by buffing player $j$ to be identical to player $i$. Then, by
  honesty, $\pi^i_j(P) \leq \pi^i_j(P')$, by $(a)$, $\pi^i_j(P')=\pi^j_i(P')$,
  and as $\pi^j_i(\cdot)$ does not depend on the skill of player
  $j$, $\pi^j_i(P')=\pi^j_i(P)$. Thus
  $$\pi^i_j(P) \leq \pi^j_i(P)\text{ for any $i<j$ and $P\in\mathcal{D}_n$.}$$

  The idea now is that, for a given $P\in\mathcal{D}_n$, we can interpret the
  probabilities $\pi^i_j(P)$ in terms of a
  fractional arc flow. For any $i, j\in[n]$
  we define $m'_{ij} = \pi^i_j(P)$. Then $\pi_j(P) = \frac1n \sum_{i=1}^n m'_{ij}$.
  Now, this does not necessarily define an arc flow as $m'_{ij}$ might be
  positive
  even if $i<j$, and we might have $m'_{12}>\frac{1}{2}$ (which is really just a
  special case of the former). However, as $m'_{ij}\leq m'_{ji}$ whenever $i<j$,
  we can cancel out these ``backwards flows'' by, whenever $m'_{ij}=x>0$ for
  $i<j$, reducing $m'_{ij}$ and $m'_{ji}$ and increasing $m'_{ii}$ and $m'_{jj}$,
  all by $x$. Let $(m_{ij})$ be the resulting matrix. Then this is an arc flow.
  As the cancelling does not change the net influx to each vertex, we have
  $\pi_j(P) = \frac1n \sum_{i=1}^n m_{ij}$. Hence the theorem follows by Lemma
  \ref{lem:faf}.
\end{proof}
  
\section{Tournament maps}\label{sect:maps}

As we have seen earlier in the article, an $n$-player tournament induces a map $P\mapsto\bm{wv}(P)$ from $\mathcal{M}_n$ to the set $\mathcal{P}_n$ of probability distributions on
$[n]$. The aim of this section is to see how honest and symmetric tournaments
can be characterized in terms of these maps.

We define an $n$-player \emph{tournament map} as any continuous function $f$
from $\mathcal{M}_n$ to $\mathcal{P}_n$. For any $M\in \mathcal{M}_n$
we denote $f(M) = (f_1(M),\dots,\,f_n(M))$.
Similarly to tournaments, we define:
\\
\\
  {\sc Symmetry:}  For any permutation $\sigma\in\mathcal{S}_n$ and any
  $P\in \mathcal{M}_n$, we define $Q=(q_{ij})\in\mathcal{M}_n$ by
  $q_{\sigma(i)\sigma(j)} = p_{ij}$ for all $i, j \in [n]$. We say that a tournament
  map $f$ is \emph{symmetric} if, for any
  $P\in \mathcal{M}_n$, $\sigma\in\mathcal{S}_n$ and any $i\in[n]$, we have
  $f_i(P)=f_{\sigma(i)}(Q).$
\\
\\
  {\sc Honesty:} A tournament map $f$ is (strictly) \emph{honest} if for any two
  distinct $i, j\in [n]$ we have that $f_i(P)$ is (strictly) increasing in
  $p_{ij}$.
\\
\\
Using these definitions it follows that the tournament map $f_{{\bm{T}}}$
induced by a
tournament $\bm{T}$ inherits the properties of $\bm{T}$.
\begin{lemma}\label{lemma:symesymhonehon}
  The tournament map induced by any symmetric tournament is symmetric. The
  tournament map induced by any honest tournament is honest.
\end{lemma}
\begin{proof} The first statement is the definition of a symmetric tournament.
  The second statement follows from Lemma \ref{lem:honesty}.
\end{proof}

We now want to show a converse to this lemma. Here we have to be a bit careful
though. Consider for instance the $2$-player tournament map
$$f_1(P) := \frac{1}{2}+\sin(p_{12}-\frac{1}{2}), \;\;\;\;
f_2(P) := \frac{1}{2}-\sin(p_{12}-\frac{1}{2}).$$
This can be shown to be symmetric and honest, but as $f_1$ and $f_2$ are not
polynomials in the entries of $P$, this map cannot be induced by any tournament
whatsoever. On the other hand, for any tournament map $f$, we can construct a
tournament $\bm{T}_f$ whose win vector approximates $f$ arbitrarily well.

\begin{definition}\label{def:Tf}
  Let $f$ be an $n$-player tournament map and let $N$ be a (large) positive
  integer. We let $\bm{T}_f=\bm{T}_{f,N}$ denote the tournament defined as
  follows: \begin{itemize} \item Play $N$ iterations of round-robin.
  \item Let $\hat{p}_{ij}$ denote the fraction of matches that $i$ won against $j$, and let $\hat{P}\in\mathcal{M}_n$ be the corresponding matrix.
    \item Randomly elect a tournament winner from the distribution given by $f(\hat{P})$.
    \end{itemize}
\end{definition}

\begin{proposition}\label{prop:tournamentmaptournament}
  Let $f$ be an $n$-player tournament map. For any $\varepsilon>0$ there exists
  an $N_0$ such that for $N\geq N_0$, the tournament $\bm{T}_f$ satisfies
  $\abs{\pi_i(P)-f_i(P)} < \varepsilon$ for all $P\in \mathcal{M}_n$ and all
  $i\in [n]$. Moreover $\bm{T}_f$ is symmetric if $f$ is symmetric, and (strictly) honest
  if $f$ is (strictly) honest.
\end{proposition}
\begin{proof}
  It is easy to see that this tournament is symmetric if $f$ is so, and
  likewise for honesty. It only remains to show that the win vector is
  sufficiently close to $f(P)$ for all $P\in\mathcal{M}_n$. First, note that
  $\pi_i(P) = \mathbb{E}f_i(\hat{P})$. Hence, by Jensen's inequality,
  $$\abs{\pi_i(P)-f_i(P)} \leq \mathbb{E}\abs{f_i(\hat{P})-f_i(P)}.$$ As $f$ is
  continuous and $\mathcal{M}_n$ is compact, $f$ is uniformly continuous.
  Hence, given $\varepsilon>0$, there exists a
  $\delta>0$ such that, for any $P$,
  $\abs{f_i(\hat{P})-f_i(P)} < \varepsilon/2$ whenever
  $\|\hat{P}-P\|_\infty < \delta$. Choosing $N_0$ sufficiently large, we can
  ensure that $\mathbb{P}(\|\hat{P}-P\|_\infty \geq \delta) < \varepsilon/2$,
  by the Law of Large Numbers.
  As, trivially, $\abs{f_i(\hat{P})-f_i(P)} \leq 1$, it follows that
  $$\mathbb{E}\abs{f_i(\hat{P})-f_i(P)} < \varepsilon/2 \cdot \mathbb{P}(\|\hat{P}-P\|_\infty < \delta) + 1\cdot \mathbb{P}(\|\hat{P}-P\|_\infty \geq \delta) \leq \varepsilon/2 + \varepsilon/2.$$
\end{proof}

For a given $\varepsilon>0$, we say that two $n$-player tournaments $\bm{T}_1$ and $\bm{T}_2$ are \emph{$\varepsilon$-close} if, for any $i\in[n]$ and $P\in\mathcal{M}_n$, we have $\abs{\pi_i(\bm{T}_1, P)-\pi_i(\bm{T}_2, P)} < \varepsilon$. A nice implication of the above results is that Definition \ref{def:Tf} provides an almost general construction of symmetric, honest tournaments in the following sense.

\begin{corollary}\label{cor:andersnormalform}
Any symmetric and honest tournament $\bm{T}$ is $\varepsilon$-close to a tournament $\bm{T}_f$ for a symmetric and honest tournament map $f$. As a consequence any such $\bm{T}$ is $\varepsilon$-close to a symmetric and honest tournament where
\begin{itemize}
\item the match schedule is fixed,
\item each pair of players meet the same number of times,
\item the tournament satisfies a stronger form of honesty, namely, given
  the outcomes of all both past and future matches in the tournament, it is never better to lose the current match than to win it.
\end{itemize}
\end{corollary}
\begin{proof}
Let $f$ be the induced tournament map of $\bm{T}$. Then $f$ is symmetric and honest by Lemma \ref{lemma:symesymhonehon}, and by Proposition \ref{prop:tournamentmaptournament}, $\bm{T}_f=\bm{T}_{f, N}$ is $\varepsilon$-close to $\bm{T}$ for $N$ sufficiently large. It is clear that $\bm{T}_f$ has the claimed properties.
\end{proof}

Let $A_n$ denote the set of all vectors $f(P)$ attained by symmetric
and honest $n$-player tournament maps $f$ at doubly monotonic
matrices $P \in \mathcal{D}_n$. 

\begin{corollary}\label{cor:weakaAn}
  $\bar{A}_n = \mathcal{A}_n$, where $\bar{A}_n$ denotes the closure of
  $A_n$.
\end{corollary}
\begin{proof}
  If $\bm{T}$ is a symmetric and honest $n$-player tournament then, by
  Lemma \ref{lemma:symesymhonehon}, the
  tournament map $f_{\bm{T}}$ induced by $\bm{T}$ is also
  symmetric and honest. For any $P \in \mathcal{M}_n$, we have 
  $\bm{wv}(\bm{T}, \, P) = f_{\bm{T}}(P)$. It follows that
  $\mathcal{A}_n \subseteq \bar{A}_n$. Conversely, for any symmetric and
  honest tournament map $f$ and any doubly monotonic matrix
  $P\in\mathcal{D}_n$, we know by Proposition
  \ref{prop:tournamentmaptournament} that there exist symmetric and honest
  tournaments $\bm{T}_f$, whose win vector at $P$ approximates $f(P)$
  arbitrarily
  well. Hence $\mathcal{A}_n$ is dense in $\bar{A}_n$. As both sets are closed,
  they must be equal. 
\end{proof}

It turns out that $A_n$ is a closed set, hence $A_n = \mathcal{A}_n$,
a fact which will be established in
Subsection \ref{subsect:polytope} below. Before that, we consider two other
applications of the above material. 

\subsection{Strictly honest tournaments}\label{subsect:strict} As has been remarked earlier in the article, the constructions of symmetric and honest tournaments presented in Sections \ref{sect:threeplayer} and \ref{sect:nplayer} are generally \emph{not} strictly honest. Since, in practice, honestly attempting to win a match typically requires a greater expenditure of effort than not trying, it is natural to require that a tournament should be strictly honest as to guarantee a strictly positive payoff for winning. We will now show how the proof of Corollary \ref{cor:andersnormalform} can be modified such that the tournament $\bm{T}_f$ is also strictly honest. Hence, any symmetric and honest tournament can be approximated arbitrarily well by symmetric and strictly honest ones.

Given $\bm{T}$, let $g = g_{\bm{T}}$ be the induced tournament map and let $h$ be
  any symmetric and strictly honest tournament map whatsoever, for instance
  $$h_i(M) := \frac{1}{{n \choose 2}} \sum_{j\neq i} m_{ij}.$$ Then
  $f=(1-\frac{\varepsilon}{2})g + \frac{\varepsilon}{2} h$ is a symmetric and
  strictly honest tournament map such that, for any $P \in \mathcal{M}_n$, 
  $$||f(P)-g(P)||_{\infty} \leq \frac{\varepsilon}{2} ||g(P)-h(P)||_{\infty} \leq \frac{\varepsilon}{2}.$$
  
  By Proposition \ref{prop:tournamentmaptournament}, we know that choosing $N$
  sufficiently large ensures that, for any $P \in \mathcal{M}_n$,
  $||\bm{wv}((\bm{T}_{f, \, N}, \, P)) - f(P)||_{\infty} < \frac{\varepsilon}{2}$. Hence $\bm{T}_{f, N}$ is $\varepsilon$-close to $\bm{T}$. On the other hand, as $f$ is strictly honest, so is $\bm{T}_{f, N}$, as desired.

\subsection{Tournaments with rounds}\label{subsect:rounds}

In our definition of ``tournament'' we required that matches be played
one-at-a-time. Many real-world tournaments consist of
``rounds'' of matches, where matches in the same round are
in principle meant to be played simoultaneously. In
practice, things usually get even more complicated, with each
round being further subdivided into non-temporally overlapping
segments, for reasons usually having to do with TV viewing. Our formal
definition of tournament is easily extended to accomodate this much complexity:
simply replace ``matches'' by ``rounds of matches'', where each player plays
at most one match per round. In defining honesty, it then makes sense
to condition both on the results from earlier rounds and
on the pairings for the current round.
\par If $\bm{T}$ is such a ``tournament with rounds'', then
there is a canonical associated tournament without rounds
$\bm{T}^{\prime}$, got by internally ordering the matches of each round uniformly
at random. It is easy to see that \par (a) $\bm{T}$ symmetric $\Leftrightarrow$
$\bm{T}^{\prime}$ symmetric,
\par (b) $\bm{T}^{\prime}$ (strictly) honest $\Rightarrow$ $\bm{T}$ (strictly)
honest.
\\
\par The reverse implication in (b) does not always hold, a
phenomenon which will be familiar to sports fans{\footnote{For example, many
    professional European football leagues currently require that, in the final
    round of the season, all matches kick off at the same time. The same rule
    applies to the final round of group matches in major international
    tournaments such as the World Cup and European Championships and was
    introduced after the so-called ``Disgrace of Gij\'{o}n'': \texttt{https://en.wikipedia.org/wiki/Disgrace$\underline{\;}$of$\underline{\;}$Gijon}}}. A toy
counterexample with four players is presented below.
\par Nevertheless, a tournament with rounds also induces a tournament map and, using the same proof idea as Lemma \ref{lem:honesty}, one can show that the induced tournament
map of any symmetric and honest tournament with rounds is symmetric and honest.
Hence, by Corollary \ref{cor:weakaAn}, any win vector that can be attained
by a symmetric and honest tournament with rounds for a doubly monotonic matrix
is contained in $\mathcal{A}_n$. In fact, for any $\varepsilon>0$, Proposition
\ref{prop:tournamentmaptournament} implies that any symmetric and honest
tournament with rounds is $\varepsilon$-close to a regular (i.e. one without
rounds) symmetric and honest tournament $\bm{T}_f$.
\\
\\
  {\sc Example 6.2.1.} Consider the following tournament with rounds $\bm{T}$:
  \\
  \\
  \emph{Step 0:} Pair off the players uniformly at random. Say the pairs are
$\{i,\,j\}$ and $\{k,\,l\}$.
\\
\emph{Round 1:} Play matches $\{i,\,j\}$ and $\{k,\,l\}$.
\\
\emph{Round 2:} Play the same matches.
\\
\emph{Step 3:} Toss a fair coin. The winner of the tournament is determined as
follows:
\par If heads, then
\par \hspace{0.5cm} - if $k$ and $l$ won one match each, the loser of the
first match between $i$ and $j$ wins the tournament
\par \hspace{0.5cm} - otherwise, the winner of the first $\{i,\,j\}$ match
wins the tournament.
\par If tails, then same rule except that we interchange the roles of the
pairs $\{i,\,j\} \leftrightarrow \{k,\,l\}$.
\\
\par It is clear that $\bm{T}$
is symmetric and honest (though not strictly honest, since what one does in
Round 2 has no effect on one's own probability of winning the tournament).
Without loss of generality, take player $i$. If he loses in Round $1$,
then he wins the tournament with probability $p_{kl}(1-p_{kl})$. If he wins in
Round 1, then he wins the tournament with probability
$\frac{1}{2} (p_{kl}^{2} + (1-p_{kl})^2)$. The latter expression is bigger for
any $p_{kl}$, and strictly so if $p_{kl} \neq \frac{1}{2}$. However, consider
any instance of $\bm{T}^{\prime}$. Without loss of generality, $i$ and $j$ play
first in Round 1. Suppose $p_{ij} > \frac{1}{2}$ and $j$ wins this match. Then
each of $k$ and $l$ would be strictly better off if they lost their first
match.

\subsection{$A_n = \mathcal{A}_n$ is a finite union of convex polytopes}\label{subsect:polytope}   
We already know that $\mathcal{A}_n$ is a convex polytope for $n = 1, 2, 3$ and, if
Conjecture \ref{conj:nplayer} holds, then this is true in general. 
In this subsection, we extend the ideas of tournament maps to show that $\mathcal{A}_n$ is a finite union of convex polytopes. We will here take \emph{convex polytope} to mean a set in $\mathbb{R}^n$ for some $n$ that can be obtained as the convex hull of a finite number of points. Equivalently, it is a bounded region of $\mathbb{R}^n$ described by a finite number of non-strict linear inequalities. In particular, a convex polytope is always a closed set. As a corollary, we show the stronger version of Corollary \ref{cor:weakaAn} that $A_n = \mathcal{A}_n$. In particular, for any $n\geq 1$, this gives the alternative characterization
\begin{equation}
\mathcal{A}_n = \{f(P) : f\text{ is a symmetric and honest $n$-player tournament map}, \, P\in\mathcal{D}_n\}
\end{equation}
of the closure of the set of achievable win vectors.

For any $P\in\mathcal{M}_n$, we define
\begin{equation}
A_n(P) = \{f(P) : f\text{ is a symmetric and honest $n$-player tournament map}\}.
\end{equation}
By definition, 
\begin{equation}\label{eq:ulgyAn}
A_n = \bigcup_{P \in \mathcal{D}_n} {A}_n (P)
\end{equation}
and so, by Corollary \ref{cor:weakaAn}, 
\begin{equation}\label{eq:infiniteP}
\mathcal{A}_n = \bar{A}_n = \overline{\bigcup_{P\in\mathcal{D}_n} {A}_n(P)}.
\end{equation}
Our strategy will consist of two main steps. First, we show that it
suffices to take the union in \eqref{eq:ulgyAn} and therefore also in \eqref{eq:infiniteP} over a finite number of
$P\in\mathcal{D}_n$. Second, for any such $P$ we give a discretization argument
that shows that ${A}_n(P)$ is a convex polytope. As then $A_n$ is a finite union of closed sets, it is closed. Hence $\mathcal{A}_n=A_n$ (without closure).

Let us begin with the first step. For any two matrices $P, Q\in \mathcal{M}_n$,
we say that $P$ and $Q$ are \emph{isomorphic} if
$p_{ij}<p_{kl} \Leftrightarrow q_{ij}<q_{kl}$. As there are only a finite number
of ways to order $n^2$ elements, the number of isomorphism classes is clearly
finite.
\begin{proposition}
If $P$ and $Q$ are isomorphic, then ${A}_n(P) = {A}_n(Q)$.
\end{proposition}
\begin{proof}
  Let $B=\{p_{ij} : i, j\in[n]\}$ and $C=\{q_{ij} : i, j\in[n]\}$. As the entries
  of $P$ and $Q$ are ordered in the same way, the sets $B$ and $C$ contain
  the same number of elements. Moreover, as each set contains $\frac{1}{2}$ and
  is invariant under the map $x\mapsto 1-x$, each contains an odd number of
  elements. Let us enumerate these by $b_0 < b_1 < \dots < b_{2k}$ and
  $c_0 < c_1 < \dots < c_{2k}$. Then $b_k=c_k=\frac{1}{2}$ and
  $b_i+b_{2k-i}=c_i+c_{2k-i}=\frac{1}{2}$. We define
  $\varphi:[0, 1]\rightarrow[0, 1]$ to be the unique piecewise-linear
  function satisfying $\varphi(0)=0$, $\varphi(b_i)=c_i$ for
  all $0\leq i \leq 2k$,  $\varphi(1)=1$.
  It follows that $\varphi$ is a continuous
  increasing function such that $\varphi(1-x) = 1-\varphi(x)$ for all
  $x\in[0, 1]$. Hence, by letting $\varphi$ act on $P\in\mathcal{M}_n$
  coordinate-wise, we can consider $\varphi$ as an increasing map from
  $\mathcal{M}_n$ to itself such that $\varphi(P)=Q$.

  Now, for any symmetric and honest tournament map $f$, it follows that
  $f \circ \varphi$ and $f \circ \varphi^{-1}$ are also symmetric and honest
  tournament maps. Moreover $f(Q)=(f\circ\varphi)(P)$ and
  $f(P)=(f\circ\varphi^{-1})(Q)$. Hence the same win vectors are achievable for
  $P$ and $Q$, as desired.
\end{proof}

As for the second step, we want to show that for any fixed $P\in\mathcal{M}_n$, ${A}_n(P)$ is a convex polytope. Given $P$, we define $B_P$ as the set consisting
of $0, 1$ and all values
$p_{ij}$ for $i, j\in[n]$. We define $\mathcal{M}_n(P)$ as the set of all
matrices $Q\in\mathcal{M}_n$ such that $q_{ij}\in B_P$ for all $i, j\in[n]$,
and define a $P$-discrete tournament map as a function from
$\mathcal{M}_n(P)$ to
$\mathcal{P}_n$. We define symmetry and honesty in the same way as for regular
tournament maps. Let ${A}_n'(P)$ be the set of all vectors $f(P)$ for
$P$-discrete, symmetric and honest $n$-player tournament maps.

\begin{proposition}
For any $P\in\mathcal{M}_n$, ${A}'_n(P)$ is a convex polytope.
\end{proposition}
\begin{proof}
  As $\mathcal{M}_n(P)$ is a finite set, we can represent any $P$-discrete
  $n$-player tournament map as a vector in a finite-dimensional (more precisely
  $\left(\abs{\mathcal{M}_n(P)}\times n\right)$-dimensional) space. The
  conditions that the map is symmetric and honest can be expressed as a finite
  number of linear equalities and non-strict inequalities to be
  satisfied by this vector. It is also clearly bounded, as it is contained in $\mathcal{M_n}\times \mathcal{P}_n$, which is a bounded set. Hence, the set of $P$-discrete, symmetric and
  honest $n$-player tournament maps form
  a convex polytope. Evaluating a tournament map at $P$ can be
  interpreted as a projection of the corresponding vector, hence
  ${A}_n'(P)$ is a
  linear projection of a convex polytope, which means that it must be a
  convex polytope itself.
\end{proof}

\begin{proposition}
For any $P\in\mathcal{M}_n$, ${A}_n(P)={A}_n'(P)$.
\end{proposition}
\begin{proof}
  As the restriction of any symmetric and honest tournament map $f$ to
  $\mathcal{M}_n(P)$ is a symmetric and honest $P$-discrete tournament map, it
  follows  that ${A}_n(P)\subseteq {A}'_n(P)$. To prove that
  ${A}'_n(P)\subseteq {A}_n(P)$, it suffices to show that any
  symmetric and honest $P$-discrete tournament map $f$ can be extended to a
  symmetric and honest (non-discrete) tournament map $g$.

  Given $Q\in\mathcal{M}_n$, we construct a random matrix
  $\mathbf{R}\in\mathcal{M}_n(P)$ as follows: for each pair of players
  $\{i, j\}$, if $q_{ij}$, and thereby also $q_{ji}$ are contained in $A_P$, let
  $\mathbf{r}_{ij}=q_{ij}$ and $\mathbf{r}_{ji}=q_{ji}$.
  Otherwise, write
  $q_{ij} = pa_k+(1-p)a_{k+1}$ for $p\in(0, 1)$ where $a_k, a_{k+1}$ denote
  consecutive elements in $A_P$ and, independently for each such pair of
  players, put $\mathbf{r}_{ij}=a_k$, $\mathbf{r}_{ji}=1-a_k$ with probability
  $p$, and
  $\mathbf{r}_{ij}=a_{k+1}$, $\mathbf{r}_{ji}=1-a_{k+1}$ with probability $1-p$.
  We define $g(Q) = \mathbb{E}f(\mathbf{R})$. This construction is clearly
  continuous and symmetric, and a simple coupling argument shows that $g_i(Q)$
  is increasing in $q_{ij}$, thus $g$ is honest. Moreover, by construction
  $g(P)=f(P)$. Hence ${A}'_n(P)\subseteq {A}_n(P)$, as desired. 
\end{proof}

\section{Futile Tournaments}\label{sect:futile}

Recall the notations $\pi_{i}^{+}$, $\pi_{i}^{-}$ in the definition of honest
tournaments in Section \ref{sect:defi}. In words, they were the probabilities
of $i$ winning the tournament, conditioned on whether $i$ won or lost a given
match and given the results of earlier matches and knowledge of the rules of
the tournament. Honesty was the criterion that $\pi_{i}^{+} \geq \pi_{i}^{-}$
should always hold. We now consider a very special case:
\begin{definition} With notation as above, a tournament is said to be
  \emph{futile} if $\pi_{i}^{+} = \pi_{i}^{-}$ always holds.
\end{definition}
  One natural way to try to reconcile the (arguably) paradoxical fact that
  symmetric and honest tournaments can benefit a worse player over a better
  one is to imagine the winning probability of player $i$ to be divided into
  two contributions. First, the result of matches where player $i$ is involved,
  where, by honesty, a higher ranked player should always be better off.
  Second, the result of matches where $i$ is not involved, where there
  is no immediate reason a player with low rank could not benefit the most.

  Following this intuition, it would make sense to expect the most unfair
  symmetric and honest tournaments to be ones without the first contribution,
  that is, symmetric and futile tournaments. However, as the following result
  shows, this is not the case.

  \begin{proposition}\label{prop:futile}
    If $\bm{T}$ is a symmetric and futile $n$-player tournament, then
    $\pi_1 = \dots = \pi_n = \frac{1}{n}$ in any specialization.
  \end{proposition}

  \begin{proof}
  Since $\bm{T}$ is futile, it is honest and hence, by
  Lemma \ref{lem:honesty}, $\pi_i$ is increasing in $p_{ij}$ at every point
  of $\mathcal{M}_{n}$, for any $i \neq j$. But
      consider the tournament $\bm{T}^c$ which has the same rules as
      $\bm{T}$, but where we reverse the result of every match. Clearly
      this will also be futile, hence honest, and corresponds to
      a change of variables $p_{uv} \mapsto 1 - p_{uv} (= p_{vu})$. Hence,
      for $i \neq j$ it
      also holds that $\pi_i$ is decreasing in $p_{ij}$ at every point of
      $\mathcal{M}_{n}$ and so $\pi_i$ does not depend on $p_{ij}$ for any
      $i\neq j$.
  
      Given a matrix $P\in\mathcal{M}_n$, we say that a player $i>1$ is a
      \emph{clone} of player $1$ if $p_{1j}=p_{ij}$ for all $j\in [n]$. Clearly,
      $\pi_1=\pi_2=\dots =\pi_n = \frac{1}{n}$ for any matrix $P$ with $n-1$
      clones of player $1$. We show by induction that the same equality holds
      for any number of clones.

      Assume $\pi_1=\pi_2=\dots =\pi_n = \frac{1}{n}$ whenever
      $P\in\mathcal{M}_n$
      contains $k\geq 1$ clones of player $1$. Let $P\in\mathcal{M}_n$ be a
      matrix that contains $k-1$ such clones. For any player $i> 1$ that is not
      a clone, we can make it into one by modifying the entries in the
      $i$:th row and column of $P$ appropriately. By
      futility, $\pi_i$ does not depend on these entries, but by the induction
      hypothesis, $i$ gets winning probability $\frac{1}{n}$ after the
      modification. Hence any $i>1$ that is not a clone of player $1$ has
      winning probabiltity $\frac{1}{n}$. By symmetry, any clone must have the
      same winning probability as player $1$, which means that these also must
      have winning probability $\frac{1}{n}$.
\end{proof}

\section{Final Remarks}\label{sect:final}

In this paper we have taken a well-established mathematical model for
tournaments - whose key ingredient is the assumption of fixed probabilities
$p_{ij}$ for player $i$ beating $j$ in a single match - and introduced and
rigorously defined three new concepts: symmetry, honesty and fairness. Our
main insight is that it is possible for a tournament to be symmetric and
(strictly) honest, yet unfair. We'd like to finish here with some remarks
on the concepts themselves.
\par Symmetry seems to us a rather uncontroversial idea. It is of course true
that, in practice, many tournaments have special arrangements which break
symmetry in a myriad of ways. Hoewever, if one wishes to develop some
general mathematical theory,
it seems like a natural restriction to impose at the beginning.
\par Turning to honesty, the fact that ``it takes effort to try and actually win
a match'' suggests that it would be more realistic to demand that the
differences $\pi_{i}^{+} - \pi_{i}^{-}$ are bounded away from zero somehow. The
same fact indicates that a more realistic model should incorporate the
possibility of there being intrinsic value for a player in trying
to minimize the total number of matches he expects to play in the tournament.
This basically involves abandoning the assumption that the $p_{ij}$ are
constant. Incorporating ``effort expended'' into our framework is therefore
clearly a non-trivial task, which we leave for future investigation.
\par Thirdly, we turn to fairness.
Various alternative notions of ``fairness'' can already be
gleaned from the existing literature. Basically, however, there are two
opposite directions from which one might criticize our definition:
\begin{itemize}
\item On the one hand, one might say we are too restrictive in only
  concentrating on the probabilities of actually winning the tournament.
  In practice, many tournaments end up with a partial ordering of the
  participants (though usually with a single maximal element), and rewards
  in the form of money, ranking points etc. are distributed according to one's
  position in this ordering. Hence, instead of defining fairness in terms
  of winning probability, one could do so in terms of expected depth
  in the final partial ordering, or some other proxy for expected reward. This
  is another possibility for future work.
\item At the other end of the spectrum, one could suggest that a fair
  tournament should not just give the best player the highest probability of
  winning, but that this probability should be close to one. There are a
  number of important papers in the literature which take this point of view,
  see for example \cite{Ben}, \cite{Bra} and \cite{Fei}. These authors
  are concerned with a different kind of question than us, namely how efficient
  (in terms of expected total number of matches played) can one make
  the tournament while ensuring that the best player wins with high
  probability ? There are elegant, rigorous results for the special case of
  the model in which $p_{ij} = p$ for all $i < j$, and some fixed
  $p \in (0, \, 1]$. Moreover, as the papers \cite{Bra} and \cite{Fei} show,
    this kind of question has applications far beyond the world of
    sports tournaments. In this regard, see also \cite{BT}, where the focus
    is more on efficiently producing a correct ranking of \emph{all}
    participants with high probability.
\end{itemize}
$\;$ \par
Since our main result is a ``negative'' one, it seems
reasonable to ask
whether there is something stronger than honesty, but still a natural
condition, which if imposed on a tournament ensures fairness, in the sense
we defined it. Of course, Schwenk's paper already gives \emph{some} kind of
positive answer: the simplest way to ensure honesty is by having
single-elimination and his method of Cohort Randomized Seeding (CRS) introduces
just sufficient randomness to ensure fairness. Note that,
since a partial seeding remains, his tournaments are not symmetric. Our
question is whether there is a natural condition which encompasses a
significantly wider range of symmetric tournaments. 
\\
\par An alternative viewpoint is to ask for ``more realistic'' examples
of tournaments which are symmetric and honest but unfair. It may be
surprising at first glance that the tournaments $\bm{T}_1$ and
$\bm{T}_2$ in Section \ref{sect:threeplayer}
are indeed unfair, but it is probably not going out on a limb to guess that
no major sports competition is ever likely to adopt those formats. This
is even more the case with the tournaments in Section \ref{sect:nplayer}, which
have the feeling of being ``rigged'' to achieve just the desired outcome.
\par As noted in Section \ref{sect:intro}, there are at least two commonly
occurring examples of symmetric and honest (and fair) tournaments:
\par - round-robin, with ties broken uniformly at random, 
\par - single-elimination with uniformly randomized seeding.\\
\noindent On the other hand, the popular two-phase format of first playing round-robin in order to rank the players for a knock-out tournament using standard seeding is symmetric but not necessarily honest (or fair). Here it's worth noting that a two-phase tournament consisting of round-robin followed by CRS single-elimination, while symmetric, need not be honest either. Suppose we have $2^k$ players, for some large $k$, all but one of whom are clones (see the proof of Proposition \ref{prop:futile}), while the last player is much worse than the clones. Suppose that, before the last round of matches in the round-robin phase, the poor player has defied the odds and won all of his $2^k - 2$ matches to date, while nobody else has won significantly more than $2^{k-1}$ matches. In that case, the poor player is guaranteed to be in the highest cohort, so it is in the interest of every clone to end up in as low a cohort as possible, as this will increase their chances of meeting the poor player in the second phase, i.e.: of having at least one easy match in that phase. In particular, it will be in the interest of a clone to lose their last round-robin match.
\par If we employ uniform randomization in the knockout phase, then the round-robin phase serves no purpose whatsoever. We do not know if there is any other randomization procedure for single-elimination which, combined with round-robin, still yields a symmetric and honest tournament.
\\
\par These observations suggest that finding ``realistic'' examples of symmetric and honest, but unfair tournaments may not be easy. Then again, sports tournaments, or even \emph{tournaments} as defined in this paper, represent a very narrow class of what are usually called ``games''. As mentioned in Section \ref{sect:intro}, a \emph{truel} could be considered as another type of game which is symmetric and honest, yet unfair (in particular, it is possible to define those terms precisely in that context). As a final speculation, we can ask whether the ``real world'' provides any examples of phenomena analogous to those considered in this paper? A social scientist might use a term like ``equal treatment'' instead of ``symmetry'', so we are asking whether the real world provides examples of situations where participants are treated equally, there is no incentive for anyone to cheat, and yet the outcome is unfair (on average). 

\vspace*{1cm}

\section*{Acknowledgements}

We thank Jan Lennartsson, Allan Schwenk and Johan W\"{a}stlund for helpful
discussions.


\end{document}